\documentclass[reqno]{amsart}

\usepackage{amsmath}
\usepackage{amsfonts}
\usepackage{amssymb,enumerate}
\usepackage{amsthm}
\usepackage[all]{xy}
\usepackage{hyperref}

\newtheorem{lem}{Lemma}[section]
\newtheorem{cor}[lem]{Corollary}
\newtheorem{prop}[lem]{Proposition}
\newtheorem{thm}[lem]{Theorem}

\newtheorem{Defn}[lem]{Definition}
\newtheorem{Ex}[lem]{Example}
\newtheorem{Question}[lem]{Question}
\newtheorem{Property}[lem]{Property}
\newtheorem{Properties}[lem]{Properties}
\newtheorem{Discussion}[lem]{Remark}
\newtheorem{Construction}[lem]{Construction}
\newtheorem{Fact}[lem]{Fact}
\newtheorem{Notation}[lem]{Notation}
\newtheorem{Subprops}{}[lem]
\newtheorem{Para}[lem]{}

\newenvironment{defn}{\begin{Defn}\rm}{\end{Defn}}
\newenvironment{ex}{\begin{Ex}\rm}{\end{Ex}}

\newenvironment{para}{\begin{Para}\rm}{\end{Para}}
\newenvironment{disc}{\begin{Discussion}\rm}{\end{Discussion}}

\newenvironment{notation}{\begin{Notation}\rm}{\end{Notation}}
\newenvironment{fact}{\begin{Fact}\rm}{\end{Fact}}

\newtheorem{intthm}{Theorem}

\newcommand{\comp}[1]{\widehat{#1}}

\newcommand{\pd}{\mathrm{pd}}

\newcommand{\ext}{\mathrm{Ext}}

\newcommand{\id}{\mathrm{id}}

\newcommand{\HH}{\mathrm{H}}
\newcommand{\Hom}{\mathrm{Hom}}	

\newcommand{\coker}{\mathrm{Coker}}

\newcommand{\tor}{\mathrm{Tor}}

\newcommand{\ol}{\overline}

\newcommand{\im}{\operatorname{Im}}

\newcommand{\xra}{\xrightarrow}
\newcommand{\xla}{\xleftarrow}
\newcommand{\shift}{\mathsf{\Sigma}}

\newcommand{\cone}{\operatorname{Cone}}

\newcommand{\Ker}{\operatorname{Ker}}

\newcommand{\wti}{\widetilde}

\newcommand{\cat}[1]{\mathcal{#1}}
\newcommand{\catx}{\cat{X}}
\newcommand{\caty}{\cat{Y}}
\newcommand{\catm}{\cat{M}}
\newcommand{\catv}{\cat{V}}
\newcommand{\catw}{\cat{W}}

\newcommand{\catp}{\cat{P}}

\newcommand{\cati}{\cat{I}}
\newcommand{\cata}{\cat{A}}
\newcommand{\catab}{\mathcal{A}b}

\newcommand{\catgi}{\cat{GI}}
\newcommand{\catgp}{\cat{GP}}

\newcommand{\catgic}{\cat{GI}_C}

\newcommand{\catgpc}{\cat{GP}_C}

\newcommand{\catac}{\cat{A}_C}
\newcommand{\catbc}{\cat{B}_C}

\newcommand{\catic}{\cat{I}_C}
\newcommand{\catpc}{\cat{P}_C}

\newcommand{\opg}{\cat{G}}

\newcommand{\catpd}[1]{\cat{#1}\text{-}\pd}
\newcommand{\xpd}{\catpd{X}}

\newcommand{\wpd}{\catpd{W}}

\newcommand{\gpd}{\catpd{GP}}
\newcommand{\gid}{\catid{GI}}
\newcommand{\catid}[1]{\cat{#1}\text{-}\id}
\newcommand{\yid}{\catid{Y}}
\newcommand{\vid}{\catid{V}}

\newcommand{\aext}{\ext_{\cata}}
\newcommand{\ahom}{\Hom_{\cata}}
\newcommand{\xaext}{\ext_{\catx\!\cata}}
\newcommand{\gwaext}{\ext_{\gw\cata}}

\newcommand{\ayext}{\ext_{\cata\caty}}
\newcommand{\agvext}{\ext_{\cata\gv}}
\newcommand{\avext}{\ext_{\cata\catv}}
\newcommand{\awext}{\ext_{\cata\catw}}
\newcommand{\waext}{\ext_{\catw \cata}}
\newcommand{\pmext}{\ext_{\catp}}
\newcommand{\miext}{\ext_{\cati}}
\newcommand{\wacomp}{\varepsilon_{\catw \cata}}

\newcommand{\avcomp}{\varepsilon_{\cata \catv}}
\newcommand{\xwacomp}{\vartheta_{\catx \catw \cata}}
\newcommand{\gwwacomp}{\vartheta_{\gw \catw \cata}}
\newcommand{\ayvcomp}{\vartheta_{\cata \caty \catv}}
\newcommand{\agvvcomp}{\vartheta_{\cata \gv \catv}}

\newcommand{\finrescat}[1]{\operatorname{res}\comp{\cat{#1}}}
\newcommand{\proprescat}[1]{\operatorname{res}\wti{\cat{#1}}}
\newcommand{\finrescatx}{\finrescat{X}}

\newcommand{\finrescatw}{\finrescat{W}}

\newcommand{\proprescatx}{\proprescat{X}}

\newcommand{\proprescatw}{\proprescat{W}}
\newcommand{\fincorescat}[1]{\operatorname{cores}\comp{\cat{#1}}}
\newcommand{\propcorescat}[1]{\operatorname{cores}\wti{\cat{#1}}}

\newcommand{\fincorescaty}{\fincorescat{Y}}
\newcommand{\fincorescatv}{\fincorescat{V}}
\newcommand{\fincorescatw}{\fincorescat{W}}

\newcommand{\propcorescaty}{\propcorescat{Y}}
\newcommand{\propcorescatv}{\propcorescat{V}}
\newcommand{\propcorescatw}{\propcorescat{W}}
\newcommand{\cpltrescat}[1]{\operatorname{res}\ol{\cat{#1}}}
\newcommand{\cpltcorescat}[1]{\operatorname{cores}\ol{\cat{#1}}}
\newcommand{\cpltrescatw}{\cpltrescat{W}}
\newcommand{\cpltcorescatw}{\cpltcorescat{W}}
\newcommand{\cpltcorescatv}{\cpltcorescat{V}}

\newcommand{\watate}{\smash{\comp{\ext}}_{\catw \cata}}
\newcommand{\awtate}{\smash{\comp{\ext}}_{\cata \catw}}
\newcommand{\avtate}{\smash{\comp{\ext}}_{\cata \catv}}
\newcommand{\pmtate}{\smash{\comp{\ext}}_{\catp}}
\newcommand{\mitate}{\smash{\comp{\ext}}_{\cati}}

\newcommand{\gw}{\opg(\catw)}
\newcommand{\gv}{\opg(\catv)}
\newcommand{\gwpd}{\catpd{G(W)}}
\newcommand{\gvid}{\catid{G(V)}}

\renewcommand{\geq}{\geqslant}
\renewcommand{\leq}{\leqslant}
\renewcommand{\ker}{\Ker}
\renewcommand{\hom}{\Hom}

\numberwithin{equation}{lem}

\begin{document}

\bibliographystyle{amsplain}

\author{Sean Sather-Wagstaff}

\address{Sean Sather-Wagstaff, Department of Mathematics \#2750,
North Dakota State University,
PO Box 6050,
Fargo, ND 58108-6050, 
USA}
\email{Sean.Sather-Wagstaff@ndsu.edu}
\urladdr{http://math.ndsu.nodak.edu/faculty/ssatherw/}
\thanks{Sean Sather-Wagstaff is supported in part by a grant from the NSA}

\author{Tirdad Sharif}
\address{Tirdad Sharif, School of Mathematics, Institute for Research in Fundamental Sciences (IPM), P.O. Box 19395-5746, Tehran Iran}
\email{sharif@ipm.ir}
\urladdr{http://www.ipm.ac.ir/IPM/people/personalinfo.jsp?PeopleCode=IP0400060}
\thanks{Tirdad Sharif is supported by a grant from IPM, (No. 83130311).}

\author{Diana White}
\address{Diana White, Department of Mathematical \& Statistical Sciences
University of Colorado Denver
Campus Box 170
P.O. Box 173364
Denver, CO 80217-3364 USA}
\email{Diana.White@ucdenver.edu}
\urladdr{http://www.math.cudenver.edu/\~{}diwhite/}

\title{Tate cohomology  with respect to semidualizing modules}

\date{\today}


\begin{abstract}
We investigate Tate cohomology of modules over a commutative noetherian ring with respect to semidualizing modules. We identify classes of modules admitting Tate resolutions and analyze the interaction  between the corresponding relative and Tate cohomology modules.  As an application of our approach, we  prove a general balance result for Tate cohomology. Our results are based on an analysis of Tate cohomology in abelian categories.
\end{abstract}

\maketitle

\section*{Introduction} \label{sec0}

This paper investigates Tate cohomology of objects
in abelian categories, inspired by the work of 
Avramov and Martsinkovsky~\cite{avramov:aratc}
and building from our own work~\cite{sather:crct,sather:gcac,sather:sgc}.
Much of our motivation comes from certain categories of modules
over a commutative ring $R$. For this introduction, we focus on 
this specific situation. (All rings in this paper
are commutative with  identity, and all modules are unital.)

An $R$-module $C$ is \emph{semidualizing} if
$R\cong \Hom_R(C,C)$ and $\ext^{\geq 1}_R(C,C)=0$.
(See Section~\ref{sec8} for background information about these modules.)
For example, the free module $R$ is semidualizing, as is a dualizing module.

Each semidualizing $R$-module $C$ comes equipped with
a certain number of  classes of $R$-modules that have good homological
properties with respect to $C$. One example is the class of
\emph{$C$-projective} $R$-modules $\catpc(R)$, consisting of the modules of the form
$P\otimes_R C$ for some projective $R$-module $P$. 
Another  example is the class $\opg(\catpc(R))$, containing the modules that
are built  by taking \emph{complete resolutions}
by modules in $\catpc(R)$.  
Other examples are the categories of modules $M$ that admit a bounded resolution
by modules from $\catpc(R)$ or from $\opg(\catpc(R))$; these are the modules $M$
with $\catpc\text{-}\pd_R(M)<\infty$ or $\opg(\catpc)\text{-}\pd_R(M)<\infty$.
For example, when $C=R$, the modules in 
$\opg(\catpc(R))$ are the \emph{Gorenstein projective $R$-modules}, and
$\opg(\catpc)\text{-}\pd_R(M)$ is the \emph{Gorenstein projective dimension} of $M$.

The first step in constructing a theory of Tate cohomology with respect 
to $C$ is to identify the modules $M$ that
admit appropriate resolutions: A
\emph{Tate $\catpc$-resolution} of $M$
is a diagram of chain maps  $T\to W\to M$
where $T$ and $W$ are certain chain complexes of modules from $\catpc(R)$.
The complexes $T$ and $W$ contain slightly different homological information about
$M$. For instance, $W$ is a resolution of $M$ which measures
$\opg(\catpc)\text{-}\pd_R(M)$ and $\catpc\text{-}\pd_R(M)$. 
The following result characterizes the modules which admit Tate $\catpc$-resolutions. 
It is contained in Theorem~\ref{cor666}.

\begin{intthm} \label{thmD}
Let $R$ be a commutative ring, and let $C$ be a semidualizing $R$-module.
An $R$-module $M$ admits a Tate $\catpc$-resolution
if and only if $\opg(\catpc)\text{-}\pd_R(M)$ is finite.
\end{intthm}

Given an $R$-module $M$ with a Tate $\catpc$-resolution
$T\to W\to M$,
one uses the complex $W$
to define the
\emph{relative cohomology} functors
$\ext_{\opg(\catpc)}^n(M,-)$ and $\ext_{\catpc}^n(M,-)$.
The complex $T$ is 
used to define the
\emph{Tate cohomology} functors $\smash{\comp{\ext}}_{\catpc}^n(M,-)$.
These cohomology functors are connected by
the following result; 
it is proved in~\eqref{prfA}, and the dual result is  Corollary~\ref{thmA'}.
The special case where $C=R$ and $M$ is finitely generated
is in~\cite[(7.1)]{avramov:aratc}.

\begin{intthm} \label{thmA}
Let $R$ be a commutative ring, and let $C$ be a semidualizing $R$-module.
Let  $M$ and $N$ be $R$-modules, and assume that
$d=\opg(\catpc)\text{-}\pd_R(M)<\infty$.
There is a long exact sequence that is natural in $M$ and $N$
\begin{align*}
0 \to
& \ext_{\opg(\catpc)}^1(M,N) \to
 \ext_{\catpc}^1(M,N) \to
\smash{\comp{\ext}}_{\catpc}^1(M,N) \to 
  \\
 \to
& \ext_{\opg(\catpc)}^2(M,N) \to
 \ext_{\catpc}^2(M,N) \to
\smash{\comp{\ext}}_{\catpc}^2(M,N) \to 
  \\
\cdots \to
& \ext_{\opg(\catpc)}^d(M,N) \to
 \ext_{\catpc}^d(M,N) \to
 \smash{\comp{\ext}}_{\catpc}^d(M,N) \to 
0
\end{align*}
and there are isomorphisms  
$\ext_{\catpc}^n(M,N) \xra{\cong}
\smash{\comp{\ext}}_{\catpc}^n(M,N)$
for each $n>d$.
\end{intthm}

The next result
shows how Tate cohomology 
detects the finiteness of  $\catpc$-projec-tive dimension.
The proof is in~\eqref{prfB}; see also Corollary~\ref{thmB'}.

\begin{intthm} \label{thmB}
Let $R$ be a commutative ring, and let $C$ be a semidualizing $R$-module.
For an $R$-module $M$ with
$\opg(\catpc)\text{-}\pd_R(M)<\infty$,
the next conditions are equivalent:
\begin{enumerate}[\quad\rm(i)]
\item \label{thmB1}
$\catpc\text{-}\pd_R(M)<\infty$;
\item \label{thmB3}
$\smash{\comp{\ext}}_{\catpc}^n(-,M)=0$ 
for each (equivalently, for some) $n\in\mathbb{Z}$;
\item \label{thmB2}
$\smash{\comp{\ext}}_{\catpc}^n(M,-)=0$ 
for each (equivalently, for some) $n\in\mathbb{Z}$;
and
\item \label{thmB6}
$\smash{\comp{\ext}}_{\catpc}^0(M,M)=0$.
\end{enumerate}
\end{intthm}

The following balance result
is another one of our main theorems; it is proved in~\eqref{prfC}. 
Corollary~\ref{Gbalance02'} shows how
it improves upon a result of Asadollahi and Salarian~\cite[(4.8)]{asadollahi:ctgim}.
It also compliments work of Iacob~\cite[Thm.\ 2]{iacob:rbgtc}
and implies some of the main results of~\cite{sather:mfhdsdm};
see Corollary~\ref{holm1}.

\begin{intthm} \label{thmC}
Let $R$ be a commutative ring, and let $B$ and $C$ be  semidualizing $R$-modules
such that $B$ is in $\catgpc(R)$.
Set $B^{\dagger}=\Hom_R(B,C)$.
Let  $M$ and $N$ be $R$-modules
such that 
$\opg(\catp_B)\text{-}\pd_R(M)<\infty$
and
$\opg(\cati_{B^{\dagger}})\text{-}\id_R(N)<\infty$.
Then there are isomorphisms
for each $n\geq 1$
$$\smash{\comp{\ext}}_{\catp_B}^{n}(M, N)\cong
\smash{\comp{\ext}}_{\cati_{B^{\dagger}}}^{n}(M, N).$$
If $R$ is noetherian  and $C$ is dualizing for $R$, 
this isomorphism
holds for all $n\in\mathbb Z$.
\end{intthm}

We conclude this section by summarizing the contents of this paper.
Section~\ref{sec1}  contains notation and background information
on the relevant subcategories of abelian categories.
Section~\ref{sec8} specifies the examples
arising from semidualizing modules.
Section~\ref{sec7} focuses on the  main properties
of Tate resolutions; it contains the proof of Theorem~\ref{thmD}. 
In Section~\ref{sec4}, we investigate the fundamental properties of Tate cohomology
and prove Theorem~\ref{thmA}.
Section~\ref{sec10} analyzes the vanishing behavior of these functors
and contains the proof of Theorem~\ref{thmB}.
Finally, Section~\ref{sec6} deals with balance for Tate cohomolgy
including the proof of Theorem~\ref{thmC}.

\section{Categories, Resolutions, and Relative Cohomology}\label{sec1}

We begin with some notation and terminology for use throughout this paper.

\begin{defn} \label{notation01}
Throughout this work
$\cata$ is an abelian category, and $\catab$ is the category of abelian groups.
Write $\catp=\catp(\cata)$ and $\cati=\cati(\cata)$ 
for the subcategories of projective and injective
objects in $\cata$, respectively.
We use the term ``subcategory'' to mean a ``full and additive  
subcategory that is closed under isomorphisms.''
A subcategory $\catx$ of $\cata$  is
\emph{exact} if it is closed under  direct summands and extensions;
it satisfies the \emph{two-of-three property}
when it is 
closed under extensions,
kernels of epimorphisms, and
cokernels of monomorphisms.
\end{defn}

\begin{defn} \label{notation01a}
We fix subcategories $\catx,\caty,\catw,\catv\subseteq\cata$ such that
$\catw\subseteq\catx$
and $\catv \subseteq\caty$.
Write $\catx\perp\caty$ 
if $\aext^{\geq1}(X,Y)=0$ for each object $X\in\catx$ and each object $Y\in\caty$.
For an object $M\in \cata$, write $M\perp\caty$ (resp., $\catx\perp M$)
if $\aext^{\geq1}(M,Y)=0$ for each object  $Y\in\caty$
(resp., if $\aext^{\geq1}(X,M)=0$ for each object  $X\in\catx$).
We say that $\catw$ is a \emph{cogenerator} for $\catx$ if,
for each object $X\in\catx$, there is an exact sequence
$$0\to X\to W\to X'\to 0$$
with $W\in\catw$ and $X'\in \catx$; and
$\catw$ is an \emph{injective cogenerator} for $\catx$ if 
$\catw$ is a cogenerator for $\catx$ such that $\catx\perp\catw$.
The terms \emph{generator} and \emph{projective generator} are
defined dually.
\end{defn}

\begin{defn} \label{notation07}
An \emph{$\cata$-complex} is a sequence of 
homomorphisms in $\cata$
$$M =\cdots\xra{\partial^M_{n+1}}M_n\xra{\partial^M_n}
M_{n-1}\xra{\partial^M_{n-1}}\cdots$$
such that $ \partial^M_{n-1}\partial^M_{n}=0$ for each integer $n$.
We frequently (and without warning)
identify objects in $\cata$ with complexes concentrated in degree 0.

Fix an integer $i$ and an $\cata$-complex $M$.
The
$i$th \emph{homology object} of $M$ is
$\HH_i(M)=\Ker(\partial^M_{i})/\im(\partial^M_{i+1})$.
The $i$th \emph{suspension} (or \emph{shift}) of
$M$, denoted $\shift^i M$, is the complex with
$(\shift^i M)_n=M_{n-i}$ and $\partial_n^{\shift^i M}=(-1)^i\partial_{n-i}^M$.
We set $\shift M=\shift^1 M$.
The \emph{hard truncation} $M_{\geq i}$ is the complex
$$M_{\geq i}=\cdots\xra{\partial^M_{i+2}}M_{i+1}\xra{\partial^M_{i+1}}
M_{i}\to 0$$
and the hard truncations $M_{> i}$, $M_{\leq i}$, and $M_{< i}$ are defined similarly.
\end{defn}

\begin{defn} \label{notation07a}
Let $M$ and $N$ be $\cata$-complexes.
The  \emph{Hom-complex} $\ahom(M,N)$ is the complex of abelian groups defined as
$\ahom(M,N)_n=\prod_p\ahom(M_p,N_{p+n})$
with $n$th differential $\partial_n^{\ahom(M,N)}$ given by
$\{f_p\}\mapsto \{\partial^{N}_{p+n}f_p-(-1)^nf_{p-1}\partial^M_p\}$.
A \emph{morphism} from $M$ to $N$
is an element of $\ker(\partial_0^{\ahom(M,N)})$;
it  is \emph{null-homotopic} if 
it is in $\im(\partial_1^{\ahom(M,N)})$.
The identity morphism $M\to M$ is denoted $\id_M$.
The complex $M$ is \emph{$\ahom(\catx,-)$-exact} if 
$\ahom(X,M)$ is exact for each object $X\in\catx$.  
The term \emph{$\ahom(-,\catx)$-exact} is
defined dually.

Fix morphisms of $\cata$-complexes $\alpha, \alpha'\colon M\to N$.
We say that
$\alpha$ and $\alpha'$ are \emph{homotopic} if the difference
$\alpha-\alpha'$ is null-homotopic.  The morphism $\alpha$ is a
\emph{homotopy equivalence} if there is a morphism
$\beta\colon N\to M$ such that 
$\beta \alpha $ is homotopic to $\id_{M}$ and
$\alpha\beta$ is homotopic to $\id_{N}$.
The complex $M$ is \emph{contractible} if 
$\id_M$ is null-homotopic.  

For each integer $i$, the morphism  $\alpha$
induces a morphism on homology objects
$\HH_i(\alpha)\colon\HH_i(M)\to\HH_i(N)$, and $\alpha$ is a
\emph{quasiisomorphism} when each $\HH_i(\alpha)$ is an isomorphism.
The \emph{mapping cone} of $\alpha$ is the complex
$\cone(\alpha)$ defined as
$\cone(\alpha)_n=N_n\oplus M_{n-1}$
with $n$th differential 
$\partial^{\cone(\alpha)}_n 
= \Bigl(\begin{smallmatrix}\partial_{n}^{N} & \alpha_{n-1} \\ 0 & -\partial_{n-1}^{M}
\end{smallmatrix} \Bigr)$.
\end{defn}

\begin{fact} \label{disc0001}
Let $\alpha\colon M\to N$ be a morphism of $\cata$-complexes.
There is a degreewise split exact sequence
$0\to \shift^{-1} N\to\shift^{-1}\cone(\alpha)\to M\to 0$
of $\cata$-complexes.
The complex $\cone(\id_M)$ is contractible.

If $M$ is contractible, then it is exact and for every $\cata$-complex $L$,
the complexes
$\ahom(M,L)$ and $\ahom(L,M)$ are exact.
\end{fact}

\begin{defn} \label{notation03}
Let $X$ be an $\cata$-complex. It is \emph{bounded} if $X_n=0$ for $|n|\gg 0$. 

Assume that
$X_{-n}=0=\HH_n(X)$ for all $n>0$ and that $M\cong\HH_0(X)$.  The natural morphism
$X\to M$ is a quasiisomorphism.  
If each $X_n$ is  in $\catx$, then
$X$ is an
\emph{$\catx$-resolution} of $M$, and
the associated exact sequence 
$$X^+ = \cdots\xra{\partial^X_{2}}X_1
\xra{\partial^X_{1}}X_0\to M\to 0$$ 
is the \emph{augmented
$\catx$-resolution} of $M$ associated to $X$.
Sometimes we call the quasiisomorphism $X\xra\simeq M$ 
a resolution of $M$.

An $\catx$-resolution $X$ is \emph{proper} if
$X^+$ is $\ahom(\catx,-)$-exact. We set
$$\proprescatx= 
\text{the subcategory of objects of $\cata$ admitting a proper 
$\catx$-resolution.}$$
The \emph{$\catx$-projective dimension} of $M$ is the quantity
$$\xpd(M)=\inf\{\sup\{n\geq 0\mid X_n\neq 0\}\mid \text{$X$ is an
$\catx$-resolution of $M$}\}.$$ 
The objects of $\catx$-projective dimension 0 are
exactly the objects of $\catx$.
We set
$$\finrescatx=
\text{the subcategory of objects $M\in\cata$ with $\xpd(M)<\infty$.}$$
One checks readily that $\proprescatx$ and $\finrescatx$ are subcategories
of $\cata$ that contain
$\catx$.

We define \emph{(proper) $\caty$-coresolutions} and \emph{$\caty$-injective dimension}
dually.  The \emph{augmented 
$\caty$-coresolution} associated to a $\caty$-coresolution $Y$ is denoted $^+Y$,
and the $\caty$-injective dimension of $M$ is $\yid(M)$.
We set
\begin{align*}
\propcorescaty
&= 
\text{the subcategory of objects of $\cata$ admitting a proper 
$\caty$-coresolution}\\
\fincorescaty
&= 
\text{the subcategory of objects $N\in\cata$ with $\yid(N)<\infty$}
\end{align*}
which  are subcategories
of $\cata$ that contain $\caty$.
\end{defn}

Auslander and Buchweitz~\cite[(1.1)]{auslander:htmcma}
provide the next important constructions.

\begin{defn}\label{ab01}
Assume that $\catx$ and $\caty$ are exact and that $\catw$ and $\catv$ are 
closed under direct summands.
Assume that $\catw$ is a cogenerator for $\catx$ and that
$\catv$ is a generator for $\caty$, and
fix an object $M\in\finrescatx$.  There exist
exact sequences in $\cata$
\begin{equation*} 
0\to K \to X_0 \to M \to 0
\qquad \qquad \qquad
0\to M\to K' \to X' \to 0
\end{equation*}
such that $K,K'\in\finrescatw$ and
$X_0,X'\in\catx$.
The first sequence  is a
\emph{$\catw\catx$-approxima-tion} of $M$, and
the second sequence  is a
\emph{$\catw\catx$-hull} of $M$.
It follows that $M$ admits a \emph{bounded strict $\catw\catx$-resolution},
that is, a bounded $\catx$-resolution $X\xra{\simeq} M$ such that
$X_i\in\catw$ for each $i\geq 1$.  This resolution is obtained
by splicing a bounded $\catw$-resolution of $K$ with the 
$\catw\catx$-approximation.

Similarly, an object $N$ in $\fincorescaty$ admits  a \emph{bounded strict $\caty\catv$-coresolution},
that is, a bounded $\caty$-coresolution $N\xra{\simeq} Y$ such that
$Y_i\in\catv$ for each $i\leq -1$. 
\end{defn}

\begin{defn} \label{rel02}
Let $f\colon M\to M'$ and $g\colon N\to N'$
be morphisms in $\cata$.
If $M$ admits a proper $\catw$-resolution 
$W\xra{\gamma} M$, then for each integer $n$
the \emph{$n$th relative cohomology group} $\waext^n(M,N)$ is
$$\waext^n(M,N)=\HH_{-n}(\ahom(W,N)).$$
If $M'$ also admits a proper $\catw$-resolution 
$W'\xra{\gamma} M'$, then~\cite[(1.8.a)]{sather:gcac} 
yields a lifting $\ol{f}\colon W\to W'$ of $f$ that is unique up to homotopy, and 
we define
\begin{gather*}
\waext^n(f,N)=\HH_{-n}(\ahom(\ol{f},N))\colon\waext^n(M',N)\to\waext^n(M,N) \\
\waext^n(M,g)=\HH_{-n}(\ahom(W,g))\colon\waext^n(M,N)\to\waext^n(M,N'). 
\end{gather*}
We write $\waext^{\geq1}(M,\caty)=0$ if $\waext^{\geq1}(M,Y)=0$
for each object $Y\in\caty$.  When $\catx\subseteq\proprescatw$,
we write $\waext^{\geq1}(\catx,\caty)=0$ if $\waext^{\geq1}(X,Y)=0$
for each object $X\in\catx$ and each object $Y\in\caty$.

When $N$ and $N'$ admit proper $\catv$-coresolutions, 
the \emph{$n$th relative cohomology group} 
$\awext^n(M,N)$ is defined dually,
as are the maps
\begin{gather*}
\avext^n(f,N)\colon\avext^n(M',N)\to\avext^n(M,N) \\
\avext^n(M,g)\colon\avext^n(M,N)\to\avext^n(M,N')
\end{gather*}
and similarly for the conditions $\avext^{\geq1}(\catx,N)=0$ and
$\avext^{\geq1}(\catx,\caty)=0$.
\end{defn}

\begin{defn} \label{rel02a}
Let $M,N$ be objects in $\cata$. 
If $M$ admits a proper $\catw$-resolution 
$W\xra{\gamma} M$ and a proper $\catx$-resolution 
$X\xra{\gamma'} M$, let $\ol{\id_M}\colon W\to X$ be a 
lifting of the identity $\id_M\colon M\to M$,
cf.\ \cite[(1.8.a)]{sather:gcac}.
This is a quasiisomorphism
such that $\gamma=\gamma'\ol{\id_M}$.
We set
$$
\xwacomp^n(M,N)=\HH_{-n}(\ahom(\ol{\id_M},N))\colon\xaext^n(M,N)\to\waext^n(M,N). 
$$
When $N$ admits a proper $\caty$-coresolution 
and a proper $\catv$-coresolution, the map
$$
\ayvcomp^n(M,N)\colon\ayext^n(M,N)\to\avext^n(M,N)
$$
is defined similarly.
\end{defn}

\begin{fact} \label{rel02'}
Let $R$ be a commutative ring, and assume that $\catw$ is a subcategory of
$\cata=\catm(R)$. Let $M,M',N,N'$ be $R$-modules
equipped with $R$-module homomorphisms
$f\colon M\to M'$ and $g\colon N\to N'$.
If $M$ admits a proper $\catw$-resolution, then
each group $\waext^n(M,N)$
is an $R$-module.
If $M'$ also admits a proper $\catw$-resolution, then 
the maps
$\waext^n(f,N)$ and $\waext^n(M,g)$
are $R$-module homomorphisms.
Similar comments hold for $\avext$ and the maps from Definition~\ref{rel02a}.
\end{fact}

\begin{fact} \label{rel03}
The uniqueness of the liftings in~\cite[(1.8)]{sather:gcac} shows that 
\begin{align*}
\waext^n&\colon\proprescatw\times\cata\to \catab 
&\text{and}&&
\avext^n&\colon\cata\times \propcorescatv\to \catab
\end{align*}
are well-defined bifunctors, and
\begin{align*}
\xwacomp ^n&\colon\xaext^n|_{(\proprescatw\cap \proprescatx)\times\cata}\to 
\waext^n|_{(\proprescatw\cap \proprescatx)\times\cata} \\
\ayvcomp ^n&\colon\ayext^n|_{\cata \times(\propcorescatv\cap \propcorescaty)}\to 
\avext^n|_{\cata \times(\propcorescatv\cap \propcorescaty)}
\end{align*}
are well-defined natural transformations,
independent of resolutions and liftings.

One has $\xaext^{\geq 1}(\catx,-)=0$ and $\ayext^{\geq 1}(-,\caty)=0$.
There is a natural equivalence  $\xaext^0\cong\ahom$ on
$\proprescatx\times\cata$, and similarly
$\ayext^0\cong\ahom$ on
$\cata\times \propcorescaty$.
\end{fact}

We conclude this section by summarizing some aspects  of~\cite{sather:sgc}.

\begin{defn} \label{defG}
An exact complex   in $\catw$ is \emph{totally $\catw$-acyclic} if it is 
$\ahom(\catw,-)$-exact and $\ahom(-,\catw)$-exact.
Let
$\opg(\catw)$ denote the  subcategory of $\cata$ whose modules are
of the form $M\cong\coker(\partial_1^W)$ for some totally $\catw$-acyclic complex $W$
in $\catw$;
we say that $W$ is a \emph{complete $\catw$-resolution of $M$}.
\end{defn}

\begin{fact} \label{gtriv01}
A contractible $\catw$-complex is totally $\catw$-acyclic; see Fact~\ref{disc0001}.

It is straightforward to show that $\catw$ is a subcategory of
$\opg(\catw)$: if $N\in\catw$, then the complex $0\to N\xra{\id_N}N\to 0$ is a 
complete $\catw$-resolution of $N$.  

Let $M$ be an object in $\opg(\catw)$ with complete $\catw$-resolution $W$.
The hard truncation 
$W_{\geq 0}$ is a proper $\catw$-resolution of $M$ such that $W_{\geq 0}^+$ is
$\ahom(-,\catw)$-exact, and 
$W_{< 0}$ is a proper $\catw$-coresolution of $M$ such that $^+W_{< 0}$ is
$\ahom(\catw,-)$-exact.  So, one has $M\in\proprescatw\cap\propcorescatw$
and 
$\waext^{\geq1}(M,\catw)=0=\awext^{\geq1}(\catw,M)$.

Using standard arguments, one sees readily that any complete $\catw$-resolution
is $\ahom(\fincorescatw,-)$-exact
and $\ahom(-,\finrescatw)$-exact.
\end{fact}

\begin{fact} \label{sgcsummary}
Assume that $\catw\perp\catw$.  
We have $\opg(\catw)\perp\finrescatw$ and $\fincorescatw \perp \opg(\catw)$.  
The category
$\opg(\catw)$ is  exact, and $\catw$ is both an injective cogenerator and
a projective generator for $\opg(\catw)$.
If $\catw$ is closed under kernels of epimorphisms
or under cokernels of monomorphisms, then so is
$\opg(\catw)$.
See~\cite[(4.3),(4.5),(4.7),(4.11),(4.12)]{sather:sgc}.
\end{fact}

\section{Semidualizing Modules and Associated Categories} \label{sec8}

Much of the motivation for this work comes from the  module categories
discussed in this section, wherein $R$ is a commutative ring.  

\begin{defn} \label{notation08}
Let $\catm(R)$ denote the category of $R$-modules.
We  write $\catp(R)$ and $\cati(R)$ for the subcategories of projective 
$R$-modules
and injective $R$-modules.
\end{defn}

The study of semidualizing modules was initiated independently (with different names)
by Foxby~\cite{foxby:gmarm}, Golod~\cite{golod:gdagpi},
and Vasconcelos~\cite{vasconcelos:dtmc}.

\begin{defn} \label{notation08a}
An $R$-module $C$ is \emph{semidualizing} if it satisfies the following:
\begin{enumerate}[\quad(1)]
\item $C$ admits a (possibly unbounded) resolution by finite rank free $R$-modules;
\item The natural homothety map $R\to\Hom_R(C,C)$ is an isomorphism; and
\item $\ext_R^{\geq 1}(C,C)=0$.
\end{enumerate}
A finitely generated projective  $R$-module of rank 1 is semidualizing.
If $R$ is Cohen-Macaulay, then $C$  is \emph{dualizing}
if it is semidualizing and $\id_R(C)$ is finite.
\end{defn}

Over a noetherian ring, the next categories were introduced by 
Foxby~\cite{foxby:gdcmr}
when $C$ is dualizing, and 
by Vasconcelos~\cite[\S 4.4]{vasconcelos:dtmc} for arbitrary $C$, with different notation.
In the non-noetherian setting, see Holm and White~\cite{holm:fear}
and White~\cite{white:gpdrsm}.

\begin{defn} \label{notation08d}
Let $C$ be a
semidualizing $R$-module.  

The \emph{Auslander class} of $C$ is the subcategory $\catac(R)$
of $R$-modules $M$ such that 
\begin{enumerate}[\quad(1)]
\item $\tor^R_{\geq 1}(C,M)=0=\ext_R^{\geq 1}(C,C\otimes_R M)$, and
\item The natural map $M\to\Hom_R(C,C\otimes_R M)$ is an isomorphism.
\end{enumerate}
The \emph{Bass class} of $C$ is the subcategory $\catbc(R)$
of $R$-modules $N$ such that 
\begin{enumerate}[\quad(1)]
\item $\ext_R^{\geq 1}(C,M)=0=\tor^R_{\geq 1}(C,\Hom_R(C,M))$, and 
\item The natural evaluation map $C\otimes_R\Hom_R(C,N)\to N$ is an isomorphism.
\end{enumerate}
\end{defn}

Based on the work of Enochs and Jenda~\cite{enochs:gipm},
the following notions were introduced and studied in this generality by
Holm and J\o rgensen~\cite{holm:smarghd}
and White~\cite{white:gpdrsm}.

\begin{defn} \label{notation08b}
Let $C$ be a semidualizing $R$-module, and set
\begin{align*}
\catpc(R)&=\text{the subcategory of modules $M\cong P\otimes_R C$ where $P$ is $R$-projective}\\
\catic(R)&=\text{the subcategory of modules $N\cong \Hom_R(C,I)$ where $I$ is $R$-injective.}
\end{align*}
Modules in $\catpc(R)$ and $\catic(R)$ are called \emph{$C$-projective}
and \emph{$C$-injective}, respectively.

A \emph{complete $\catp\catpc$-resolution} is a complex $X$ of $R$-modules 
satisfying the following:
\begin{enumerate}[\quad(1)]
\item $X$ is exact and $\Hom_R(-,\catpc(R))$-exact; and
\item $X_i$ is projective for  $i\geq 0$ and $X_i$ is  $C$-projective for $i< 0$.
\end{enumerate}
An $R$-module $M$ is \emph{$\text{G}_C$-projective} if there
exists a complete $\catp\catpc$-resolution $X$ such that $M\cong\coker(\partial^X_1)$,
and $X$ is a \emph{complete $\catp\catpc$-resolution of $M$}.  Set
$$\catgpc(R)=\text{the subcategory of $\text{G}_C$-projective $R$-modules}.$$
In the case $C=R$ we use the more common terminology
``complete projective resolution'' and ``Gorenstein projective module''
and the notation $\catgp(R)$.

A \emph{complete $\catic\cati$-coresolution}
is a complex $Y$ of $R$-modules 
such that:
\begin{enumerate}[\quad(1)]
\item $Y$ is exact and $\Hom_R(\catic(R),-)$-exact; and
\item $Y_i$ is injective for  $i\leq 0$ and $Y_i$ is  $C$-injective for $i>0$.
\end{enumerate}
An $R$-module $N$ is \emph{$\text{G}_C$-injective} if there
exists a complete $\catic\cati$-coresolution $Y$ such that $N\cong\Ker(\partial^Y_0)$,
and $Y$ is a \emph{complete $\catic\cati$-coresolution of $N$}.
Set
$$\catgic(R)=\text{the subcategory of $\text{G}_C$-injective $R$-modules}.$$
In the case $C=R$ we use the more common terminology
``complete injective resolution'' and ``Gorenstein injective module''
and the notation $\catgi(R)$.
\end{defn}

\begin{notation} \label{notn1}
Let $C$ be a
semidualizing $R$-module.  
We  abbreviate as follows:
\begin{align*}
\pd_R(-)
&=\catp(R)\text{-}\pd(-) 
&
\id_R(-)
&=\cati(R)\text{-}\id(-) 
\\
\catpc\text{-}\pd_R(-)
&=\catpc(R)\text{-}\pd(-) 
&
\catic\text{-}\id_R(-)
&=\catic(R)\text{-}\id(-) 
\\
\opg\catp\text{-}\pd_R(-)
&=\opg\catp(R)\text{-}\pd(-) 
&
\opg\cati\text{-}\id_R(-)
&=\opg\cati(R)\text{-}\id(-)
\\
\opg\catpc\text{-}\pd_R(-)
&=\opg\catpc(R)\text{-}\pd(-) 
&
\opg\catic\text{-}\id_R(-)
&=\opg\catic(R)\text{-}\id(-)
\\
\opg(\catpc)\text{-}\pd_R(-)
&=\opg(\catpc(R))\text{-}\pd(-) 
&
\opg(\catic)\text{-}\id_R(-)
&=\opg(\catic(R))\text{-}\id(-). 
\end{align*}
\end{notation}

\begin{fact} \label{disc01}
Let $B$ and $C$ be 
semidualizing $R$-modules.  
The Auslander class $\catac(R)$ contains
every projective $R$-module and every $C$-injective $R$-module,
and the Bass class $\catbc(R)$ contains
every injective $R$-module and every $C$-projective $R$-module;
see~\cite[Lems.\ 4.1, 5.1]{holm:fear}.
These classes also satisfy the two-of-three property
by~\cite[Cor.\ 6.3]{holm:fear}.
Hence, $\catac(R)$ contains the
$R$-modules of finite projective dimension and the $R$-modules of
finite $\cati_C$-injective dimension, and
$\catbc(R)$ contains
the
$R$-modules of finite injective dimension and the $R$-modules of
finite $\catp_C$-projective dimension.
From~\cite[(2.8)]{takahashi:hasm} we know that an $R$-module $M$
is in $\catbc(R)$ if and only if $\Hom_R(C,M)\in\catac(R)$,
and that $M\in\catac(R)$ if and only if $C\otimes_RM\in\catbc(R)$.

The category $\catpc(R)$ is exact and closed under 
kernels of epimorphisms by~\cite[Prop.\ 5.1(b)]{holm:fear}
and~\cite[(2.8)]{white:gpdrsm}.
Also, $\catpc(R)$ is an injective cogenerator and
a projective generator for $\opg(\catpc(R))=\catgpc(R)\cap\catbc(R)$,
and $\opg(\catpc(R))$ is  exact and closed under kernels of epimorphisms; 
see~\cite[Secs.~4--5]{sather:sgc}.
In particular,  $\catpc(R)\perp\catpc(R)$.

The category $\catic(R)$ is exact and closed under
cokernels of monomorphisms.
Also,  $\catic(R)$ is an injective cogenerator and
a projective generator for $\opg(\catic(R))=\catgic(R)\cap\catac(R)$,
and $\opg(\catic(R))$ is  exact and closed under cokernels of monomorphisms; 
see~\cite[Secs.~4--5]{sather:sgc}.
In particular, we have $\catic(R)\perp\catic(R)$.

If $B\in\catgpc(R)$, then 
$\Hom_R(B,C)$ is also semidualizing; see, e.g. \cite[(2.11)]{christensen:scatac}.
If  $C$ is dualizing, then $B\in\catgpc(R)$ and $C\cong B\otimes_R\hom_R(B,C)$; 
see~\cite[(3.3.10)]{bruns:cmr},
\cite[(3.3)]{gerko:sdc} and~\cite[(4.4)]{white:gpdrsm}.
\end{fact}

The next lemma is from an early version of~\cite{takahashi:hasm}.
We are grateful to Takahashi and White for allowing us to include it here.
\newcommand{\cd}{C^{\dagger}}

\begin{lem} \label{lemold}
Let $R$ be a commutative ring, and let $C$ be a semidualizing $R$-module.
Assume that $R$ is Cohen-Macaulay with a dualizing module $D$,
and set $\cd=\hom_R(C,D)$. 
For each $R$-module $M$, one has
$\catpc\text{-}\pd_R(M)<\infty$
if and only if $\cati_{\cd}\text{-}\id_R(M)<\infty$.
\end{lem}

\begin{proof}
We prove the forward implication; the proof of the converse is similar.
Assume that $n=\catpc\text{-}\pd_R(M)<\infty$.
The category of modules of finite $\cati_{\cd}\text{-}\id$
satisfies the two-of-three property by~\cite[(3.4)]{takahashi:hasm}.
Hence, using a routine induction argument on $n$,
it suffices to assume that $n=0$ and prove that $\cati_{\cd}\text{-}\id_R(M)<\infty$.

So, we assume that there is a projective $R$-module $P$ such that
$M\cong C\otimes_RP$. 
This yields the second equality in the next sequence
$$
\cati_{\cd}\text{-}\id_R(M)
=\id_R(\cd\otimes_RM)
=\id_R(\cd\otimes_RC\otimes_RP)
=\id_R(D\otimes_RP)<\infty.
$$
The first equality is from~\cite[(2.11.b)]{takahashi:hasm},
and the third equality is from Fact~\ref{disc01}. 
The finiteness follows from the fact that $\id_R(D)$ is finite and $P$ is projective.
\end{proof}

The next three lemmas are for use in 
Corollary~\ref{holm1}.

\begin{lem} \label{lem1004}
Let $R$ be a commutative  ring, and let $C$ be a semidualizing $R$-module.
Let $M$ be an $R$-module.
\begin{enumerate}[\quad\rm(a)]
\item \label{lem1004a}
$M$ is in $\opg(\catpc(R))$ if and only if
$\Hom_R(C,M)$ is in $\opg\catp(R)\cap\catac(R)$.
\item \label{lem1004b}
$M$ is in $\opg(\catic(R))$ if and only if
$C\otimes_RM$ is in $\opg\cati(R)\cap\catbc(R)$.
\end{enumerate}
\end{lem}

\begin{proof}
We prove part~\eqref{lem1004b}; the proof of~\eqref{lem1004a} is dual.

Assume first that $M\in\opg(\catic(R))$, and fix a complete $\catic$-resolution $Y$ of $M$.
Fact~\ref{disc01} implies that $M\in\catac(R)$, and 
that $C\otimes_RM\in\catbc(R)$.
Since each module $Y_i$ is in $\catic(R)$, it is straightforward to show that
the complex $C\otimes_RY$ is a complex of injective $R$-modules;
see, e.g., \cite[Thm.\ 1]{holm:fear}. Also, since the modules $M\cong\ker(\partial^Y_0)$ 
and $Y_i$ are 
all in $\catac(R)$, it is straightforward to show that $C\otimes_RY$ is exact
and that $C\otimes_RM\cong\ker(\partial^{C\otimes_RY}_0)$.

By assumption, the complex $Y$ is $\Hom_R(\catic(R),-)$-exact. Hence, for each
injective $R$-module $I$, the following complex is exact
\begin{align*}
\Hom_R(\Hom_R(C,I),Y)
&\cong\Hom_R(\Hom_R(C,I),\Hom_R(C,C\otimes_RY))
\\
&\cong\Hom_R(C\otimes_R\Hom_R(C,I),C\otimes_RY)
\\
&\cong\Hom_R(I,C\otimes_RY).
\end{align*}
In this sequence, the first isomorphism comes from the fact that each $Y_i$ is in
$\catac(R)$. The second isomorphism is Hom-tensor adjointness,
and the third isomorphism is due to the condition $I\in\catbc(R)$.
It follows that $C\otimes_RY$ is a complete injective resolution of $C\otimes_RM$,
so we have $C\otimes_RM\in\opg\cati(R)$.

For the converse, assume that  $C\otimes_RM\in\opg\cati(R)\cap\catbc(R)$.
Fact~\ref{disc01} implies that $M\in\catac(R)$.
Let $Z$ be a complete injective resolution of $C\otimes_RM$.
Since the modules $C\otimes_RM$ and $Z_i$ are in $\catbc(R)$, we conclude that the complex
$\Hom_R(C,Z)$ is exact with
$M\cong\ker(\partial^{\Hom_R(C,Z)}_0)$. 
Thus, to conclude the proof, we need to show that
$\Hom_R(C,Z)$ is $\Hom_R(\catic(R),-)$-exact and $\Hom_R(-,\catic(R))$-exact.
Let $J$ be an injective $R$-module. 
Since $Z$ is $\Hom_R(\cati(R),-)$-exact, the next complex is exact
\begin{align*}
\Hom_R(J,Z)
&\cong\Hom_R(C\otimes_R\Hom_R(C,J),Z) \\
&\cong\hom_R(\Hom_R(C,J),\Hom_R(C,Z)).
\end{align*}
The isomorphisms are from the condition $J\in\catbc(R)$
and from Hom-tensor adjointness, respectively.
Thus,
$\Hom_R(C,Z)$ is $\Hom_R(\catic(R),-)$-exact.

The complex $\Hom_R(C,Z)$ consists of modules in $\catic(R)\subseteq\catac(R)$
and has $\ker(\partial^{\Hom_R(C,Z)}_0)\cong \Hom_R(C,M)\in\catac(R)$.
It follows  that $C\otimes_R\Hom_R(C,Z)$ is exact.
The fact that $J$ is injective implies that the next complex is exact
\begin{align*}
\Hom_R(C\otimes_R\Hom_R(C,Z),J)
&\cong\Hom_R(\Hom_R(C,Z),\Hom_R(C,J))
\end{align*}
where the isomorphism is from Hom-tensor adjointness.
It follows that 
$\Hom_R(C,Z)$ is  $\Hom_R(-,\catic(R))$-exact,
as
desired.
\end{proof}

The next two results improve upon Lemma~\ref{lem1004}
and compliment~\cite[(4.2), (4.3)]{holm:smarghd}.
The proof of Lemma~\ref{lem0701'} is dual to that of~\ref{lem0701}. 

\begin{lem} \label{lem0701}
Let $R$ be a commutative ring, and 
let $C$ be a semidualizing $R$-module. For an $R$-module $M$,
the following conditions are equivalent:
\begin{enumerate}[\quad\rm(i)]
\item \label{lem0701i}
$\opg(\catpc)\text{-}\pd_R(M)<\infty$;
\item \label{lem0701ii}
$\opg\catpc\text{-}\pd_R(M)<\infty$ and $M\in\catbc(R)$; and
\item \label{lem0701iii}
$\gpd_R(\hom_R(C,M))<\infty$ and $M\in\catbc(R)$.
\end{enumerate}
When these conditions are satisfied, we have
\begin{equation}\label{lem0701c}
\opg(\catpc)\text{-}\pd_R(M)
=\opg\catpc\text{-}\pd_R(M)
=\gpd_R(\hom_R(C,M)).
\end{equation}
\end{lem}

\begin{proof}
\eqref{lem0701i}$\implies$\eqref{lem0701ii}.
Assume that $\opg(\catpc)\text{-}\pd_R(M)<\infty$.
Since $\opg(\catpc(R))$ is a subcategory of $\opg\catpc(R)$,
we have 
$\opg\catpc\text{-}\pd_R(M)\leq\opg(\catpc)\text{-}\pd_R(M)<\infty$.
As $\catbc(R)$ satisfies the two-of-three property and contains
$\opg(\catpc(R))$, we have $M\in\catbc(R)$.

\eqref{lem0701ii}$\implies$\eqref{lem0701iii}.
Assume that $g=\opg\catpc\text{-}\pd_R(M)<\infty$ and $M\in\catbc(R)$.
The condition $M\in\catbc(R)$ implies that $M$ has a 
proper $\catpc(R)$-resolution
$T$ by~\cite[(2.3)]{takahashi:hasm}. 
In particular, the complex $\hom_R(C,T^+)$ is exact, and
it follows that $\hom_R(C,T)$ is a projective
resolution of $\hom_R(C,M)$ with
$$\coker(\partial^{\hom_R(C,T)}_{g+1})\cong\hom_R(C,\coker(\partial^{T}_{g+1})).$$
Since each $T_i$ is in
$\catpc(R)\subseteq\opg\catpc(R)$,
the condition $g=\opg\catpc\text{-}\pd_R(M)<\infty$ 
implies that $K_g=\coker(\partial^{T}_{g+1})$ is $\text{G}_C$-projective;
see~\cite[(2.16)]{holm:smarghd} and~\cite[(2.20)]{holm:ghd}.
Since $M$ is in $\catbc(R)$ and each $T_i$ is in $\catbc(R)$,
Fact~\ref{disc01} implies that $K_g\in\catbc(R)$.
It follows that $K_g$ is in $\catgpc(R)\cap\catbc(R)=\opg(\catpc(R))$,
so Lemma~\ref{lem1004}\eqref{lem1004a} implies that
$\hom_R(C,K_g)\in\catgp(R)$.
Hence, the exact sequence
$$0\to \hom_R(C,K_g)\to  \hom_R(C,T_{g-1})\to\cdots\to  \hom_R(C,T_0)\to  \hom_R(C,M)\to 0$$
shows that $\gpd_R(\hom_R(C,M))\leq g=\opg\catpc\text{-}\pd_R(M)<\infty$.

\eqref{lem0701iii}$\implies$\eqref{lem0701i}.
Assume that $d=\gpd_R(\hom_R(C,M))<\infty$ and $M\in\catbc(R)$.
Let $T$ be a proper $\catpc(R)$-resolution of $M$,
and set $K_d=\coker(\partial^{T}_{d+1})$.
As in the previous paragraph, 
$\hom_R(C,T)$ is a projective
resolution of $\hom_R(C,M)$ with
$$\hom_R(C,K_d)\cong
\coker(\partial^{\hom_R(C,T)}_{d+1})\in\catac(R).$$
The fact that $d=\gpd_R(\hom_R(C,M))<\infty$ implies that
$\coker(\partial^{\hom_R(C,T)}_{d+1})$ is Gorenstein projective,
and we conclude from Lemma~\ref{lem1004}\eqref{lem1004a} that
$K_d\in\opg(\catpc(R))$.
Hence the exact sequence
$0\to K_d\to T_{d-1}\to\cdots\to T_0\to 0$
shows that we have $\opg(\catpc)\text{-}\pd_R(M)\leq d=\gpd_R(\hom_R(C,M))<\infty$.

Finally, assume that conditions~\eqref{lem0701i}--\eqref{lem0701iii} are satisfied.
The proofs of the three implications
yield the inequalities in the next sequence:
\begin{align*}
\opg(\catpc)\text{-}\pd_R(M)
\leq\opg\catpc\text{-}\pd_R(M)
\leq\gpd_R(\hom_R(C,M))
\leq\opg(\catpc)\text{-}\pd_R(M).
\end{align*}
This verifies the equalities in~\eqref{lem0701c}.
\end{proof}

\begin{lem} \label{lem0701'}
Let $R$ be a commutative ring, and 
let $C$ be a semidualizing $R$-module. For an $R$-module $M$,
the following conditions are equivalent:
\begin{enumerate}[\quad\rm(i)]
\item \label{lem0701'i}
$\opg(\catic)\text{-}\id_R(M)<\infty$;
\item \label{lem0701'ii}
$\opg\catic\text{-}\id_R(M)<\infty$ and $M\in\catac(R)$; and
\item \label{lem0701'iii}
$\gid_R(C\otimes_RM)<\infty$ and $M\in\catac(R)$.
\end{enumerate}
When these conditions are satisfied, we have
\begin{xxalignat}{3}
  &{\hphantom{\square}}& \opg(\catic)\text{-}\id_R(M)
  &
  =\opg\catic\text{-}\id_R(M)
  =\gid_R(C\otimes_RM).  &&\qed
\end{xxalignat}
\end{lem}

\begin{disc}
Lemmas~\ref{lem0701}--\ref{lem0701'} have 
the following interpretations in terms of 
``Foxby equivalence''. 
Fact~\ref{disc01} shows that the functors $C\otimes_R-$ and 
$\hom_R(C,-)$ provide natural equivalences between the 
Auslander and Bass classes, as we indicate  in the middle row of the following
diagram: 
$$\xymatrix{
\catac(R)\bigcap\operatorname{res}\text{-}\widehat{\opg\catp}_{\leq n}\ar@<1ex>[rr]_{\sim} \ar@{^(->}[d]
&& \operatorname{res}\text{-}\widehat{\opg(\catpc)}_{\leq n} \ar@<1ex>[ll]\ar@{^(->}[d]
\\
\catac(R) \ar@<1ex>[rr]^-{C\otimes_R-}_{\sim}
&&\catbc(R) \ar@<1ex>[ll]^-{\hom_R(C,-)}
\\
\operatorname{cores}\text{-}\widehat{\opg(\catic)}_{\leq n}\ar@<1ex>[rr]_{\sim} \ar@{^(->}[u]
&& \catbc(R) \bigcap\operatorname{cores}\text{-}\widehat{\opg\cati}_{\leq n}. \ar@<1ex>[ll]\ar@{^(->}[u]
}$$
The equivalences in the top and bottom rows of the diagram follow 
from Lemmas~\ref{lem0701} and \ref{lem0701'}, using the equivalence in the middle row.
Here, the notation $\operatorname{res}\text{-}\widehat{\opg\catp}_{\leq n}$
stands for the category of $R$-modules $M$ with
$\gpd_R(M)\leq n$, et cetera.
\end{disc}

The final three results of this section are proved like Lemmas~\ref{lem1004}--\ref{lem0701'}.

\begin{lem} 
Let $R$ be a commutative  ring, and let $C$ be a semidualizing $R$-module.
Let $M$ be an $R$-module.
\begin{enumerate}[\quad\rm(a)]
\item 
$M$ is in $\opg\catp(R)\cap\catac(R)$ if and only if
$C\otimes_RM$ is in $\opg(\catpc(R))$.
\item 
$M$ is in $\opg\cati(R)\cap\catbc(R)$ if and only if
$\Hom_R(C,M)$ is in $\opg(\catic(R))$. \qed
\end{enumerate}
\end{lem}

\begin{lem} 
Let $R$ be a commutative ring, and 
let $C$ be a semidualizing $R$-module. For an $R$-module $M$,
the following conditions are equivalent:
\begin{enumerate}[\quad\rm(i)]
\item 
$\opg(\catpc)\text{-}\pd_R(C\otimes_RM)<\infty$;
\item 
$\opg\catpc\text{-}\pd_R(C\otimes_RM)<\infty$ and $M\in\catac(R)$; and
\item 
$\gpd_R(M)<\infty$ and $M\in\catac(R)$.
\end{enumerate}
When these conditions are satisfied, we have
\begin{xxalignat}{3}
  &{\hphantom{\square}}& \opg(\catpc)\text{-}\pd_R(C\otimes_RM)
  &
=\opg\catpc\text{-}\pd_R(C\otimes_RM)
=\gpd_R(M).
  &&\qed
\end{xxalignat}
\end{lem}

\begin{lem} 
Let $R$ be a commutative ring, and 
let $C$ be a semidualizing $R$-module. For an $R$-module $M$,
the following conditions are equivalent:
\begin{enumerate}[\quad\rm(i)]
\item 
$\opg(\catic)\text{-}\id_R(\hom_R(C,M))<\infty$;
\item 
$\opg\catic\text{-}\id_R(\hom_R(C,M))<\infty$ and $M\in\catbc(R)$; and
\item 
$\gid_R(M)<\infty$ and $M\in\catbc(R)$.
\end{enumerate}
When these conditions are satisfied, we have
\begin{xxalignat}{3}
  &{\hphantom{\square}}& \opg(\catic)\text{-}\id_R(\hom_R(C,M))
  &
  =\opg\catic\text{-}\id_R(\hom_R(C,M))
  =\gid_R(M).  &&\qed
\end{xxalignat}
\end{lem}

\section{Tate resolutions} \label{sec7}

In this section, we study the resolutions used to define 
our Tate cohomology functors. 
In many cases, the objects admitting such resolutions are
precisely the objects of finite $\opg(\catw)$-projective/injective dimension;
see Theorems~\ref{cplt05} and~\ref{cor666}.

\begin{defn} \label{cplt02}
Let $M$ and $N$  be  objects in $\cata$.

A \emph{Tate $\catw$-resolution} of $M$ is a diagram 
$T\xra{\alpha}W\xra{\gamma} M$
of morphisms of $\cata$-complexes
wherein $T$ is an exact $\catw$-complex that
is totally $\catw$-acyclic, $\gamma$ is a proper $\catw$-resolution of $M$,
and $\alpha_n$ is an isomorphism for $n\gg 0$.
We set
$$\cpltrescatw=\text{the subcategory of objects of $\cata$ admitting a Tate 
$\catw$-resolution.}$$

A \emph{Tate $\catv$-coresolution} of $N$ is a diagram 
$N\xra{\delta} V\xra{\beta} S$
of morphisms of $\cata$-complexes
wherein $S$ is an exact $\catv$-complex that
is totally $\catv$-acyclic, $\delta$ is a proper $\catv$-coresolution of $N$,
and $\beta_n$ is an isomorphism for $n\ll 0$. We set
$$\cpltcorescatv=
\text{the subcategory of objects of $\cata$ admitting a Tate 
$\catv$-coresolution.}$$
\end{defn}

\begin{fact} \label{cplt03}
Given  $M',M''\in \cpltrescatw$ with Tate $\catw$-resolutions
$T'\xra{\alpha'}W'\xra{\gamma'} M'$
and $T''\xra{\alpha''}W''\xra{\gamma''} M''$,
one readily shows that
the direct sum
$$T'\oplus T''\xra{\left(\begin{smallmatrix}\alpha' & 0 \\ 0 & \alpha''\end{smallmatrix}\right)}
W \oplus W''
\xra{\left(\begin{smallmatrix} \gamma' & 0 \\ 0 & \gamma''\end{smallmatrix}\right)} M \oplus M''$$
is a Tate $\catw$-resolution of $M \oplus M''$.
It follows that $\cpltrescatw$ is a subcategory
of $\cata$.
Similarly, we see that $\cpltcorescatv$ is a subcategory
of $\cata$.

If $M$ admits a Tate $\catw$-resolution
$T\to W\to M$, then $W$ is a proper $\catw$-resolution of $M$.
Hence, $\cpltrescatw\subseteq\proprescatw$, and similarly,
$\cpltcorescatv \subseteq\propcorescatv$.

If $M$ is in $\opg(\catw)$ with complete $\catw$-resolution $T$, 
then $M$ admits a Tate 
$\catw$-resolution $T\to T_{\geq 0}\to M$ and a Tate 
$\catw$-coresolution $M\to \shift T_{<0}\to \shift T$.
Hence, $\opg(\catw)$ is a subcategory of $\cpltrescatw\cap \cpltcorescatw$.

Assume that $\catw\perp\catw$. If
$M\in\finrescatw$, then any bounded $\catw$-resolution $W\xra{\gamma}M$
is proper by~\cite[(3.2.a)]{sather:gcac}, and this
yields a Tate $\catw$-resolution $0\to W\xra{\gamma}M$.
In particular, we have
$\finrescatw\subseteq\cpltrescatw$.
Similarly, if $\catv\perp\catv$, then 
$\fincorescatv\subseteq\cpltcorescatv$.
\end{fact}

The next two results are tools for the proof of Theorem~\ref{cplt05}.

\begin{lem} \label{cplt04}
One has $\cpltrescatw\subseteq\finrescat{G(W)}$
and $\cpltcorescatv \subseteq\fincorescat{G(V)}$.
\end{lem}

\begin{proof}
We prove the first containment; the proof of the second containment is dual.

Let $M$ be an object in $\cpltrescatw$, and 
fix a Tate $\catw$-resolution
$T\xra{\alpha}W\xra{\gamma} M$.
Since $T$ is an exact totally $\catw$-acyclic complex in $\catw$, the object
$\coker(\partial_n^{T})$ is in $\opg(\catw)$ for each integer $n$.
By assumption, the homomorphism $\alpha_n$ is an isomorphism
for $n\gg 0$, and it follows that
$\coker(\partial_n^{W})\cong\coker(\partial_n^{T})\in\opg(\catw)$ for $n\gg 0$.
Also, each object $W_i$ is in $\opg(\catw)$, so the exact sequence
$$0\to \coker(\partial_n^{W})\to W_{n-1}\to\cdots\to W_0\to M\to 0$$
is a bounded augmented $\opg(\catw)$-resolution of $M$.
\end{proof}

\begin{lem} \label{cplt01}
Assume that $\catx$ and $\caty$ are exact, that $\catw$ is both an injective cogenerator 
and a projective generator for $\catx$, and  that $\catv$ is both an injective cogenerator 
and a projective generator for $\caty$.
\begin{enumerate}[\quad\rm(a)]
\item \label{cplt01item1}
Let $M$ be an object in $\finrescatx$.
If $\catx$ is closed under kernels of epimorphisms,
then $M$ admits a Tate $\catw$-resolution
$T\xra{\alpha}W\xra{\gamma} M$
such that $\alpha_n$ 
is an isomorphism for each $n\geq\xpd(M)$ and each object $\Ker(\partial^T_i)$ is in $\catx$.
Moreover, this resolution can be built so that
$\alpha_n$ is a split surjection for all $n$.
\item \label{cplt01item2}
Let $N$ be an object in $\fincorescaty$.
If $\caty$ is closed under cokernels of monomorphisms,
then $N$ admits a Tate $\catv$-coresolution
$N\xra{\delta} V\xra{\beta} S$
such that $\beta_n$ 
is an isomorphism for each $n\leq-\yid(N)$ and each object $\Ker(\partial^S_i)$ is in $\caty$.
Moreover, this resolution can be built so that
$\beta_n$ is a split injection for all $n$.
\end{enumerate}
\end{lem}

\begin{proof}
We prove part~\eqref{cplt01item1}; the proof of~\eqref{cplt01item2} is dual.

Since $\catw$ is a projective generator and an injective generator for $\catx$, we have
$\catx\subseteq\proprescatw\cap\propcorescatw$ and
$\finrescatx\subseteq\proprescatw$ by~\cite[(3.3)]{sather:gcac}. 
In particular, the object $M$ admits a proper $\catw$-resolution $W\xra{\gamma}M$. 
Set $d=\xpd(M)$.  Since $\catx$ is closed under kernels of epimorphisms, 
it follows from~\cite[(3.3)]{auslander:htmcma}
that $X=\Ker(\partial_{d-1}^W)$ is  in $\catx$,
and hence $X$ admits a proper $\catw$-coresolution $X\xra{\simeq}\wti{W}$
such that each $\Ker(\partial_i^{\wti{W}})$ is in $\catx$; see~\cite[(1.8)]{sather:sgc}.
A standard argument 
using the condition $\catw\perp\catx$
shows that $^+\wti{W}$ is $\ahom(\catw,-)$-exact.

Set $\comp{W}= \shift^{d-1}\wti{W}$.
The properness of $\wti{W}$ 
yields a morphism  $\gamma\colon \comp{W}\to W_{<d}$
making the following diagram commute.
\begin{equation} 
\begin{split} \label{diag11}
\xymatrix{
0 \ar[r] & X \ar[r]\ar[d]^{\id_X} & \comp{W}_{d-1} \ar[r]\ar[d]^{\gamma_{d-1}} 
&  \cdots\ar[r] & \comp{W}_{0} \ar[r]\ar[d]^{\gamma_{0}} & \comp{W}_{-1} \ar[r]\ar[d]^{\gamma_{-1}} & \cdots \\
0 \ar[r] & X \ar[r] & W_{d-1} \ar[r] &  \cdots\ar[r] & W_{0} \ar[r] & 0 \ar[r] & \cdots
}
\end{split}
\end{equation}
The top row of this diagram is both
$\ahom(\catw,-)$-exact and $\ahom(-,\catw)$-exact.
The truncation $W_{\geq d}$ is a proper $\catw$-resolution of $X$,
hence the complex $W_{\geq d}^+$ is $\ahom(\catw,-)$-exact;
a standard argument 
using the condition $\catx\perp\catw$
shows that it is also $\ahom(-,\catw)$-exact.
Let $T^{(1)}$ be the complex obtained by splicing $W_{\geq d}$ and $\comp{W}$ along $X$.
It follows that each $T^{(1)}_n$ is in $\catw$ and that 
$T^{(1)}$ is both $\ahom(\catw,-)$-exact and $\ahom(-,\catw)$-exact.
Set
$$\alpha^{(1)}_n=\begin{cases}
\gamma_n & \text{for $n<d$} \\
\id_{W_n}& \text{for $n\geq d$.}
\end{cases}$$
The diagram~\eqref{diag11} shows that  $\alpha^{(1)}\colon T^{(1)}\to W$
is a morphism, and it follows that the diagram
$T\xra{\alpha^{(1)}}W\xra{\gamma}M$ is a Tate $\catw$-resolution.

Next we show how to modify 
the Tate $\catw$-resolution $T^{(1)}\xra{\alpha^{(1)}}W\xra{\gamma}M$ to build 
a Tate $\catw$-resolution $T\xra{\alpha}W\xra{\gamma}M$
such that each $\alpha_n$ is a split surjection and such that
$\alpha_n=\alpha_n'$ for all $n\geq d$.
To this end, it suffices to construct  a contractible $\catw$-complex $T^{(2)}$
and a morphism $\alpha\colon T^{(1)}\oplus T^{(2)}\to W$ such that
$\alpha_n$ is a split surjection for each $n< d$, and such that 
$T^{(2)}_n=0$ for each $n\geq d$.  

Consider the  truncation $W_{<d}$. 
The complex $T^{(2)}=\shift^{-1}\cone(\id_{W_{<d}})$ is contractible, and 
$T^{(2)}_n=0$ for each $n\geq d$; see Fact~\ref{disc0001}.
Let $f\colon T^{(2)}\to W$ denote the composition of the natural morphisms 
$T^{(2)}=\shift^{-1}\cone(\id_{W_{<d}})\to W_{<d}\to W$.
Note that $f_n$ is a split epimorphism for each $n<d$,
and $f_n=0$ for each $n\geq d$.
One checks readily that the morphisms
$\alpha_n=(\alpha^{(1)}_n \,\,\,\, f_n)\colon T^{(1)}_n\oplus T^{(2)}_n\to W_n$
describe a morphism of complexes satisfying the desired properties.
\end{proof}

The next result is a version of Lemma~\ref{cplt01} for objects in $\catx$
with fewer hypotheses on the categories.

\begin{prop} \label{prop0701}
Let $M$ be an object in $\cata$.
Assume that $\catw$ is an injective cogenerator for $\catx$,
and that $\catv$ is a projective generator for $\caty$.
\begin{enumerate}[\quad\rm(a)]
\item \label{prop0701a}
If $M\in\catx$, then $M\in\cpltrescatw$ if and only if $M\in\opg(\catw)$.
\item \label{prop0701b}
If $M\in\caty$, then $M\in\cpltcorescatv$ if and only if $M\in\opg(\catv)$.
\end{enumerate}
\end{prop}

\begin{proof}
We prove part~\eqref{prop0701a}; the proof of~\eqref{prop0701b} is dual.
One implication is covered by the containment
$\opg(\catw)\subseteq \cpltcorescatw$ from Remark~\ref{cplt03}.

For the converse, assume that $M$ is in $\cpltcorescatw$
and fix a Tate $\catw$-resolution $T\xra\alpha W\xra\gamma M$.
By assumption, the augmented resolution $W^+$ is $\ahom(\catw,-)$-exact.
We claim that it is also $\ahom(-,\catw)$-exact.
Indeed, since $\catw$ is an injective cogenerator for $\catx$,
we have $M\perp\catw$, and 
the condition $\catw\subseteq\catx$ implies that
$W_i\perp\catw$ for each $i\geq 0$.
A standard induction argument yields the claim.

We claim that $\catw\perp M$. As in the proof of Lemma~\ref{cplt04},
the object $\coker(\partial^W_i)$ is in $\opg(\catw)$ for all $i\gg 0$.
Hence, Fact~\ref{sgcsummary} implies that $\catw\perp\coker(\partial^W_i)$ for all $i\gg 0$.
Since $\catw\perp W_i$ for all $i$, a standard induction argument yields the claim.

Since $\catw$ is an injective cogenerator for $\catx$,
the object $M$ admits a proper $\catw$-coresolution
$M\xra\simeq\wti W$. 
Hence, the augmented coresolution $^+\wti W$ is $\ahom(-,\catw)$-exact.
A standard induction argument,
using the conditions $\catw\perp M$ and $\catw\perp \wti W_i$,
shows that  $^+\wti W$ is also $\ahom(\catw,-)$-exact.

Splice the resolutions $W$ and $\wti W$ to construct
the following exact sequence in $\catw$
$$\widehat W=\qquad \cdots
\xra{\partial^W_2} W_1
\xra{\partial^W_1} W_0
\xra{\hspace{4.5mm}} \wti W_{0}
\xra{\partial^{\wti W}_0} \wti W_{-1}
\xra{\partial^{\wti W}_{-1}}\cdots
$$
such that $M\cong\coker(\partial^{\comp W}_1)=\coker(\partial^{W}_1)$.
Since $W^+$ and $^+\wti W$ are $\ahom(\catw,-)$-exact and $\ahom(-,\catw)$-exact,
it follows that $\comp W$ is
a complete resolution of $M$, so
$M$ is in $\opg(\catw)$, by definition.
\end{proof}

The following characterizations of $\cpltrescatw$ and $\cpltcorescatv$ 
are akin to~\cite[(3.1)]{avramov:aratc}.

\begin{thm} \label{cplt05}
Assume that $\catw$ is closed under kernels of epimorphisms and that $\catw\perp\catw$.
Assume that $\catv$ is closed under cokernels of monomorphisms and $\catv\perp\catv$.
\begin{enumerate}[\quad\rm(a)]
\item \label{cplt05item1}
An object $M\in\cata$ admits a Tate $\catw$-resolution 
$T\xra\alpha W\xra\gamma M$
(such that each $\alpha_n$ is a split surjection)
if and only if
$\catpd{G(W)}(M)<\infty$. Hence, we have $\cpltrescatw=\finrescat{G(W)}$,
so the category $\cpltrescatw$ is closed under direct summands and
satisfies the two-of-three property.
\item \label{cplt05item2}
An object $N\in\cata$ admits a Tate $\catv$-coresolution 
$N\xra\delta V\xra\beta S$
(such that each $\beta_n$ is a split injection)
if and only if
$\catid{G(V)}(N)<\infty$. Hence, we have
$\cpltcorescatv=\fincorescat{G(V)}$, so
the category $\cpltcorescatv$ is closed under direct summands and
satisfies the two-of-three property.
\end{enumerate}
\end{thm}

\begin{proof}
The desired equivalences follow from Lemmas~\ref{cplt04} and~\ref{cplt01},
using Fact~\ref{sgcsummary}. The properties of $\cpltrescatw$ and $\cpltcorescatv$
follow from~\cite[(3.4),(3.5)]{auslander:htmcma}. 
\end{proof}

The next result contains Theorem~\ref{thmD} from the introduction.

\begin{thm} \label{cor666}
Let $R$ be a commutative ring, and let $C$ be a semidualizing $R$-module.
With $\catpc=\catpc(R)$ and $\catic=\catic(R)$,
one has
$\operatorname{res}\widehat{\opg(\catpc)}=\operatorname{res}\ol{\catpc}$
and
$\operatorname{cores}\widehat{\opg(\catic)}=\operatorname{cores}\ol{\catic}$.
Also, the categories $\operatorname{res}\widehat{\opg(\catpc)}$
and $\operatorname{cores}\widehat{\opg(\catic)}$ are closed under direct summands and
satisfy the two-of-three property.
\end{thm}

\begin{proof}
Fact~\ref{disc01} implies 
that $\catpc(R)$ satisfies the hypotheses of Theorem~\ref{cplt05}\eqref{cplt05item1}
and that $\catic(R)$ satisfies the hypotheses of Theorem~\ref{cplt05}\eqref{cplt05item2}.
\end{proof}

The next result is the key for well-definedness and functoriality of Tate cohomology.
The proof is almost identical to that of~\cite[(5.3)]{avramov:aratc}.

\begin{lem} \label{comp01}
Let $M,M',N,N'$ be objects in $\cata$. 
Assume that $M$ and $M'$ admit Tate $\catw$-resolutions
$T\xra{\alpha}W\xra{\gamma} M$ and
$T'\xra{\alpha'}W'\xra{\gamma'} M'$, and that 
$N$ and $N'$ admit  Tate $\catv$-coresolutions
$N\xra{\delta} V\xra{\beta} S$ and
$N'\xra{\delta'} V'\xra{\beta'} S'$
\begin{enumerate}[\quad\rm(a)]
\item \label{comp01item1}
For each morphism $f\colon M\to M'$ there is
a morphism $\ol{f}\colon W\to W'$, unique up to homotopy,
making the right-most square in the next diagram commute
$$
\xymatrix{
T \ar[r]^{\alpha} \ar[d]_{\comp{f}} & W \ar[r]^{\gamma} \ar[d]_{\ol{f}} & M \ar[d]_f \\
T' \ar[r]^{\alpha'}  & W' \ar[r]^{\gamma'}  & M'
}
$$
and for each such $\ol{f}$ there exists
a morphism $\comp{f}\colon T\to T'$, unique up to homotopy,
making the left-most square in the diagram commute up to homotopy.
If $f$ is an isomorphism, then  $\ol{f}$ and $\comp{f}$ are
homotopy equivalences.
\item \label{comp01item2}
For each morphism $g\colon N\to N'$ there is
a morphism $\ol{g}\colon V\to V'$, unique up to homotopy,
making the left-most square in the next diagram commute
$$
\xymatrix{
N \ar[r]^{\delta} \ar[d]_{\comp{g}} & V \ar[r]^{\beta} \ar[d]_{\ol{g}} & S \ar[d]_f \\
N' \ar[r]^{\delta'}  & V' \ar[r]^{\beta'}  & S'
}
$$
and for each such $\ol{g}$ there exists
a morphism $\comp{g}\colon S\to S'$, unique up to homotopy,
making the right-most square in the diagram commute up to homotopy.
If $g$ is an isomorphism, then $\ol{g}$ and $\comp{g}$ are
homotopy equivalences. \qed
\end{enumerate}
\end{lem}

What follows is a horseshoe lemma for Tate (co)resolutions 
like~\cite[(5.5)]{avramov:aratc}.  
The proof is similar to that of~\cite[(5.5)]{avramov:aratc},
but  it is different enough to merit inclusion.

\begin{lem} \label{horseshoe02}
Assume that $\catw$ is closed under kernels of epimorphisms and that $\catw\perp\catw$.
Assume that $\catv$ is closed under cokernels of monomorphisms and $\catv\perp\catv$.
\begin{enumerate}[\quad\rm(a)]
\item \label{horseshoe02item1}
Fix  an exact sequence 
$0\to M'\xra{\xi}M\xra{\zeta} M''\to 0$ in $\cata$
that is $\ahom(\catw,-)$-exact.  Assume that $M'$ and $M''$ admit Tate
$\catw$-resolutions
$T'\xra{\alpha'}W'\xra{\gamma'} M'$ and 
$T''\xra{\alpha''}W''\xra{\gamma''} M''$
such that $\alpha_n'$ and $\alpha_n''$ are 
split surjections for all $n\in\mathbb Z$
and isomorphisms
for each $n\geq d$.
Then $M$ admits a Tate
$\catw$-resolution
$T\xra{\alpha}W\xra{\gamma} M$ such that $\alpha_n$ is an  isomorphism
for each $n\geq d$ and such that 
there is a commutative diagram of morphisms
\begin{equation} 
\begin{split}
\label{diag12}
\xymatrix{
0 \ar[r] 
& T' \ar[r]^{\comp{\xi}}\ar[d]_{\alpha'} 
& T \ar[r]^{\comp{\zeta}}\ar[d]_{\alpha} 
& T'' \ar[r]\ar[d]_{\alpha''} 
& 0\\
0 \ar[r] 
& W' \ar[r]^{\ol{\xi}}\ar[d]_{\gamma'} 
& W \ar[r]^{\ol{\zeta}}\ar[d]_{\gamma} 
& W'' \ar[r]\ar[d]_{\gamma''} 
& 0 \\
0 \ar[r] & M' \ar[r]^{\xi} & M \ar[r]^{\zeta} & M'' \ar[r] & 0
}
\end{split}
\end{equation}
wherein the top two rows are degreewise split exact.
\item \label{horseshoe02item2}
Fix  an exact sequence 
$0\to N'\xra{\rho}N\xra{\tau} N''\to 0$ in $\cata$
that is $\ahom(-,\catv)$-exact.  Assume that $N'$ and $N''$ admit Tate
$\catv$-coresolutions
$N'\xra{\delta'}V'\xra{\beta'} S'$ and 
$N''\xra{\delta''}V''\xra{\beta''} S''$
such that $\beta_n'$ and $\beta_n''$ are 
split injections for all $n\in\mathbb Z$
and  isomorphisms
for each $n\leq d$.
Then $N$ admits a Tate
$\catv$-coresolution
$N\xra{\delta}V\xra{\beta} S$ such that $\beta_n$ is an  isomorphism
for each $n\leq d$ and such that 
there is a commutative diagram of morphisms
$$\xymatrix{
0 \ar[r] 
& N' \ar[r]^{\rho}\ar[d]_{\delta'} 
& N \ar[r]^{\tau}\ar[d]_{\delta} 
& N'' \ar[r]\ar[d]_{\delta''} 
& 0\\
0 \ar[r] 
& V' \ar[r]^{\ol{\rho}}\ar[d]_{\beta'} 
& V \ar[r]^{\ol{\tau}}\ar[d]_{\beta} 
& V'' \ar[r]\ar[d]_{\beta''} 
& 0 \\
0 \ar[r] & S' \ar[r]^{\comp{\rho}} & S \ar[r]^{\comp{\tau}} & S'' \ar[r] & 0
}
$$
wherein the bottom two rows are degreewise split exact.
\end{enumerate}
\end{lem}

\begin{proof} 
We prove part~\eqref{horseshoe02item1}; the proof of~\eqref{horseshoe02item2} is dual.
The lower half of the diagram~\eqref{diag12} is constructed in
the relative horseshoe lemma~\cite[(1.9.a)]{sather:gcac}.
Note that we have 
$W_n=W_n'\oplus W_n''$ for each $n\in\mathbb Z$,
and $\ol\xi_n=\left(\begin{smallmatrix}\id_{W_n'} \\ 0\end{smallmatrix}\right)$
and $\ol\zeta_n=(0 \,\,\,\, \id_{W_n''})$.
Furthermore, we have
$
\partial^W_n=\left(\begin{smallmatrix}\partial^{W'}_n & f_n \\ 0 & \partial^{W''}_n\end{smallmatrix}\right)
$
for some
$f_n\in\ahom(W_n'',W_{n-1}')$;
and the equation $\partial^W_n\partial^W_{n+1}=0$
implies that
\begin{equation} \label{diff01}
\partial^{W'}_{n}f_{n+1}+f_n\partial^{W''}_{n+1}=0.
\end{equation}
We set
$T_n=T_n'\oplus T_n''$ for each $n\in\mathbb Z$,
and $\comp\xi_n=\left(\begin{smallmatrix}\id_{T_n'} \\ 0\end{smallmatrix}\right)$
and $\comp\zeta_n=(0 \,\,\,\, \id_{T_n''})$.

The proof will be complete once we  construct morphisms 
$g_n\in\ahom(T_n'',T_{n-1}')$ and
$h_n\in\ahom(T_n'',W_{n}')$ for each $n\in\mathbb Z$
such that 
\begin{gather} \label{diff03}
\partial^{T'}_{n}g_{n+1}+g_n\partial^{T''}_{n+1}=0\\
 \label{diff02}
h_n\partial^{T''}_{n+1}=\partial^{W'}_{n+1}h_{n+1}+f_{n+1}\alpha''_{n+1}-\alpha'_ng_{n+1}.
\end{gather}
Indeed, once this is done we set
\begin{align*}
\partial^T_n
&=\left(\begin{matrix}\partial^{T'}_n & g_n \\ 0 & \partial^{T''}_n\end{matrix}\right)
&&\text{and}
&\alpha_n
&=\left(\begin{matrix}\alpha'_n & h_n \\ 0 & \alpha''_n\end{matrix}\right).
\end{align*}
Using the equation~\eqref{diff03},
it is straightforward to show  that $\partial^T$ makes $T$ into a chain complex
such that $\comp\xi$ and $\comp\zeta$ are chain maps.
Similarly, the equation~\eqref{diff02} implies that $\alpha$ is a chain map.
Since the matrices defining these maps are upper-triangular,
it follows readily that the diagram~\eqref{diag12} commutes, using
the fact that the horizontal maps in the top two rows are the canonical
injections and surjections. Since $\alpha_n'$ and $\alpha_n''$ are isomorphisms
for each $n\geq d$, the snake lemma implies that $\alpha_n$ is an isomorphism
for each $n\geq d$. Similarly, $\alpha_n'$ and $\alpha_n''$ are surjections
for each $n\in\mathbb Z$, the snake lemma implies that $\alpha_n$ is a surjection
for each $n\in\mathbb Z$. Finally, the fact that $\catw$ is closed under kernels of epimorphisms
implies that each $\ker(\alpha_n)\in\catw$; so, the condition $\catw\perp\catw$
implies that each $\alpha_n$ is a split surjection. 
Since the top row $\mathbb T$ of~\eqref{diag12} is degreewise split exact,
the sequence $\ahom(U,\mathbb T)$ is exact for each $U\in\catw$. 
Since $\ahom(U,T')$ and $\ahom(U,T'')$ are exact, a long exact
sequence argument shows that $\ahom(U,T)$ is also exact.
In summary, we conclude that $T$ is $\ahom(\catw,-)$-exact, and a similar
argument shows that it is $\ahom(-,\catw)$-exact.

The assumption that $\alpha_n'$ and $\alpha_n''$ are  isomorphisms
for each $n\geq d$ implies that
\begin{align*}
\coker(\partial^{W'}_{d+1})
&\cong\coker(\partial^{T'}_{d+1})
\in\opg(\catw)
&\text{and}&&
\coker(\partial^{W''}_{d+1})
&\cong\coker(\partial^{T''}_{d+1})
\in\opg(\catw).
\end{align*}
The exact sequence of complexes
$$0\to W'_{\geq d}\to W_{\geq d}\to W''_{\geq d}\to 0$$
has associated long exact sequence
\begin{equation} \label{exact17}
0\to \coker(\partial^{W'}_{d+1})\to\coker(\partial^{W}_{d+1})
\to\coker(\partial^{W''}_{d+1})\to 0.
\end{equation}
Fact~\ref{sgcsummary}
implies that $\opg(\catw)$ is closed under extensions,
so $\coker(\partial^{W}_{d+1}) \in \opg(\catw)$.  

For each $n> d$ set $g_n=(\alpha'_{n-1})^{-1} f_{n}\alpha''_{n}$.
For each $n> d$, this yields
\begin{align*}
g_{n}\partial^{T''}_{n+1}
&=(\alpha'_{n-1})^{-1} f_{n}\alpha''_{n}\partial^{T''}_{n+1} 
=(\alpha'_{n-1})^{-1} f_{n}\partial^{W''}_{n+1}\alpha''_{n+1} \\
&=-(\alpha'_{n-1})^{-1} \partial^{W'}_{n}f_{n+1}\alpha''_{n+1} 
=-(\alpha'_{n-1})^{-1} \partial^{W'}_{n}\alpha'_{n}(\alpha'_{n})^{-1}f_{n+1}\alpha''_{n+1} \\
&=-(\alpha'_{n-1})^{-1} \alpha'_{n-1}\partial^{T'}_{n}(\alpha'_{n})^{-1}f_{n+1}\alpha''_{n+1}  
=-\partial^{T'}_{n}(\alpha'_n)^{-1} f_{n+1}\alpha''_{n+1}\\
&=-\partial^{T'}_{n}g_{n+1}.
\end{align*}
The first, fourth, and sixth equalities are by definition;
the second one holds because $\alpha''$ is a chain map;
the third one is from equation~\eqref{diff01};
and the fifth one holds because $\alpha'$ is a chain map.
This implies that~\eqref{diff03} is satisfied
for each $n> d$.  Thus, we have constructed the complex $T_{\geq d}$
and a degreewise
split exact sequence
\begin{equation} 
\label{exact19}
0\to T'_{\geq d}\xra{\comp{\xi}_{\geq d}^{}} T_{\geq d}\xra{\comp{\zeta}_{\geq d}^{}} T''_{\geq d}\to 0.
\end{equation}
For $n\geq d$, set 
$h_n=0$.
One checks readily that our choices for $g_n$ and $h_n$ satisfy~\eqref{diff02}
for all $n> d$,
and that $\alpha_n$ is an isomorphism for $n\geq d$. In particular, we have
$\coker(\partial^{T}_{d+1})\cong\coker(\partial^{W}_{d+1})$.
The sequence~\eqref{exact17} is $\ahom(-,\catw)$-exact 
because $\aext^1(\coker(\partial^{W''}_{d+1}),\catw)=0$;
see Fact~\ref{sgcsummary}.  Hence, 
the relative horseshoe lemma~\cite[(1.9.b)]{sather:gcac} yields
a commutative diagram of morphisms
\begin{equation} 
\begin{split}
\label{diag15}
\xymatrix{
0\ar[r] 
& \coker(\partial^{T'}_{d+1}) \ar[r]\ar[d]_{\simeq} 
& \coker(\partial^{T}_{d+1}) \ar[r]\ar[d]_{\simeq} 
& \coker(\partial^{T''}_{d+1}) \ar[r]\ar[d]_{\simeq} 
& 0 \\
0\ar[r] 
& T'_{<d} \ar[r]^{\left(\begin{smallmatrix}\id_{T'_{<d}} \\ 0\end{smallmatrix}\right)} 
& T_{<d} \ar[r]^{\left(\begin{smallmatrix}0 & \id_{T''_{<d}} \end{smallmatrix}\right)} 
& T''_{<d} \ar[r] & 0 \\
}
\end{split}
\end{equation}
Splice $T_{\geq d}$ and $T_{<d}$ along $\coker(\partial^T_{d+1})$
to form $T$.
Note that the differential on  $T$ is of the form
$
\partial^T_n=\left(\begin{smallmatrix}\partial^{T'}_n & g_n \\ 0 & \partial^{T''}_n\end{smallmatrix}\right)
$
and the equation 
$\partial^T_n\partial^T_{n+1}=0$ implies
that~\eqref{diff03} holds for all $n\in\mathbb Z$.
It remains to build 
the $h_n$ for $n< d$ such that~\eqref{diff02} holds for all $n\leq d$.
We generate the remaining homomorphisms by descending induction on $n$,
for which the base case ($n> d$) has already been addressed with $h_n=0$.

By induction, we assume that $h_{n+1}$ has been constructed and we find $h_n$.
Using the fact that $T''$ is $\ahom(-,W_n')$-exact, it suffices to show that
the homomorphism
$\partial^{W'}_{n+1}h_{n+1}+f_{n+1}\alpha''_{n+1}-\alpha'_ng_{n+1}$
is a cycle in $\ahom(T'',W'_n)_{n+1}$.
This is done in the following sequence wherein the first, third, and fifth equalities are routine:
\begin{align*}
(\partial^{W'}_{n+1}h_{n+1}+f_{n+1}\alpha''_{n+1}-\alpha'_ng_{n+1})\partial^{T''}_{n+2} \hspace{-5cm} \\
&=\partial^{W'}_{n+1}h_{n+1}\partial^{T''}_{n+2}
+f_{n+1}\alpha''_{n+1}\partial^{T''}_{n+2}
-\alpha'_ng_{n+1}\partial^{T''}_{n+2} \\
&=\partial^{W'}_{n+1}(\partial^{W'}_{n+2}h_{n+2}+f_{n+2}\alpha''_{n+2}-\alpha'_{n+1}g_{n+2}) \\
&\qquad+f_{n+1}\alpha''_{n+1}\partial^{T''}_{n+2}
-\alpha'_ng_{n+1}\partial^{T''}_{n+2} \\
&=\partial^{W'}_{n+1}\partial^{W'}_{n+2}h_{n+2}
+\partial^{W'}_{n+1}f_{n+2}\alpha''_{n+2}
-\partial^{W'}_{n+1}\alpha'_{n+1}g_{n+2} \\
&\qquad+f_{n+1}\alpha''_{n+1}\partial^{T''}_{n+2}
-\alpha'_ng_{n+1}\partial^{T''}_{n+2} \\
&=0
+\partial^{W'}_{n+1}f_{n+2}\alpha''_{n+2}
-\alpha'_{n}\partial^{T'}_{n+1}g_{n+2} 
+f_{n+1}\partial^{W''}_{n+2}\alpha''_{n+2}
-\alpha'_ng_{n+1}\partial^{T''}_{n+2} \\
&=(\partial^{W'}_{n+1}f_{n+2}+f_{n+1}\partial^{W''}_{n+2})\alpha''_{n+2}
-\alpha'_{n}(\partial^{T'}_{n+1}g_{n+2} +g_{n+1}\partial^{T''}_{n+2}).
\end{align*}
The second equality follows because $h_{n+1}$ satisfies the equation~\eqref{diff02};
the fourth one follows as $\alpha'$ and $\alpha''$ are morphisms and $W'$ is a complex.
The last expression in this sequence vanishes by~\eqref{diff01} and~\eqref{diff03}.
This completes the proof.
\end{proof}

The next  result provides strict resolutions, as in~\cite[(3.8)]{avramov:aratc}, for use in
Theorem~\ref{mixed01}.
Note that Lemma~\ref{cplt01} provides Tate resolutions 
satisfying  the hypotheses.

\begin{lem} \label{l01}
Assume that $\catw$ is closed under direct summands. Let $f\colon M\to M'$
be a morphism in $\cpltrescatw$, and
let $T\xra\alpha W\xra\gamma M$ and $T'\xra{\alpha'} W'\xra{\gamma'} M'$ be  Tate $\catw$-resolutions
such that $\coker(\partial^T_1),\coker(\partial^{T'}_1)\in\catx$ and 
such that $\alpha_n$ and $\alpha'_n$ are split surjections for all $n$.
\begin{enumerate}[\quad\rm(a)]
\item\label{l01a}
There exists a degreewise split exact sequence of $\cata$-complexes
\begin{equation*}
0\to \shift^{-1}X\to \wti{T}\to W\to 0
\end{equation*}
where $\wti T=(T_{\geq 0})^+$, and
satisfying the following conditions:
\\
\begin{tabular}{lll}
$\bullet$
$X$ is a bounded strict $\catw\catx$-resolution of $M$,
&\qquad \qquad&
$\bullet$
$\wti{T}$ is exact, 
\\
$\bullet$
$\wti{T}_n=0$ for each $n<-1$, 
&\qquad&
$\bullet$
$\wti{T}_{-1}$ is in $\catx$,
\\
$\bullet$
$\wti{T}_n$ is in $\catw$ for each $n\geq 0$, and
&\qquad&
$\bullet$
$\wti{T}_{\geq 0}\cong T_{\geq 0}$.
\end{tabular}
\item\label{l01b}
There exists a commutative diagram of morphisms of $\cata$-complexes
$$\xymatrix{
0\ar[r] 
& \shift^{-1}X\ar[r]\ar[d]_{\shift^{-1}f^*} 
& \wti{T}\ar[r] \ar[d]_{\wti f} 
& W\ar[r] \ar[d]_{\ol f}
& 0
\\
0\ar[r] 
& \shift^{-1}X'\ar[r]
& \wti{T'}\ar[r]  
& W'\ar[r] 
& 0
}$$
wherein each row is an exact sequence as in part~\eqref{l01a},
the morphisms $f^*$ and $\ol f$ are lifts of $f$, and
$\wti f$ is induced by a lift of $f$.
\end{enumerate}
\end{lem}

\begin{proof}
\eqref{l01a}
The hard truncation $T_{\geq 0}$ is a proper $\catw$-resolution of
$\coker(\partial^T_1)$. Set $\wti T=(T_{\geq 0})^+$.
The morphism $\alpha\colon T\to W$ is degreewise a split surjection,
and it follows that the induced morphism
$\nu\colon \wti T\to W$ is degreewise a split surjection.
Setting $X=\shift\ker(\nu)$, yields a degreewise split exact sequence of the desired form.
Since $\wti T$ is exact,  the associated long exact sequence
shows that $X$ is a resolution of $M$.
Since $\alpha_n$ is an isomorphism for $n\gg 0$,
we conclude that $X$ is bounded. As $\alpha_n$
is a split surjection for each $n$, we have
$X_n\in\catw$ for each $n\geq 1$.
Since $X_0\cong\coker(\partial^T_1)\in\catx$,
it follows that $X$ is 
a bounded strict $\catw\catx$-resolution of $M$.

\eqref{l01b}
Lemma~\ref{comp01}\eqref{comp01item1} yields the following
commutative diagram 
$$
\xymatrix{
T \ar[r]^{\alpha} \ar[d]_{\comp{f}} & W \ar[r]^{\gamma} \ar[d]_{\ol{f}} & M \ar[d]_f \\
T' \ar[r]^{\alpha'}  & W' \ar[r]^{\gamma'}  & M'
}
$$
of morphisms of $\cata$-complexes.
Using the definitions $\wti T=(T_{\geq 0})^+$ and $\wti{T'}=(T'_{\geq 0})^+$,
it is straightforward to show that $\comp f$ induces a morphism
$\wti f\colon \wti T\to\wti{T'}$ that makes the next diagram commute
$$\xymatrix{
0\ar[r] 
& \shift^{-1}X\ar[r]
& \wti{T}\ar[r] \ar[d]_{\wti f} 
& W\ar[r] \ar[d]_{\ol f}
& 0
\\
0\ar[r] 
& \shift^{-1}X'\ar[r]
& \wti{T'}\ar[r]  
& W'\ar[r] 
& 0.
}$$
From the conditions $X=\shift\ker(\nu)$ and $X'=\shift\ker(\nu')$
it is straightforward to show that $\wti f$ induces a morphism
$f^*$ making the desired diagram commute. 

By definition, $\ol f$ is a lift of
$f$. Since $\wti{T}$ and $\wti{T'}$ are exact, the morphism $\wti{f}$ is a quasiisomorphism.
Using the induced diagrams on long exact sequences,
one readily shows that these facts imply that $f^*$ is a lift of $f$.
\end{proof}

The proof of the next result is dual to the previous proof.

\begin{lem} \label{l01'}
Assume that  $\catv$ is closed under direct summands.
Let $g\colon N\to N'$
be a morphism in $\cpltcorescatv$, and
let $N\xra\delta V\xra\beta L$ and $N'\xra{\delta'} V'\xra{\beta'} L'$ be  Tate $\catv$-coresolutions
such that $\ker(\partial^L_0),\ker(\partial^{L'}_0)\in\caty$ and 
such that $\beta_n $ and $\beta'_n$ are split surjections for all $n$.
\begin{enumerate}[\quad\rm(a)]
\item\label{l01'a}
There exists a degreewise split exact sequence of $\cata$-complexes
\begin{equation*}
0\to V\to \wti{S}\to \shift Y\to 0
\end{equation*}
where $\wti S=(S_{\geq 0})^+$, and
satisfying the following conditions:
\\
\begin{tabular}{lll}
$\bullet$
$Y$ is a bounded strict $\caty\catv$-coresolution of $N$,
&\qquad \qquad&
$\bullet$
$\wti{S}$ is exact,
\\
$\bullet$
$\wti{S}_n=0$ for each $n>1$, 
&\qquad&
$\bullet$
$\wti{S}_{1}$ is in $\caty$,
\\
$\bullet$
$\wti{S}_n$ is in $\catv$ for each $n\leq 0$, and
&\qquad&
$\bullet$
$\wti{S}_{\leq 0}\cong S_{\leq 0}$.
\end{tabular}
\item\label{l01'b}
There exists a commutative diagram of morphisms of $\cata$-complexes
$$\xymatrix{
0\ar[r] 
& V\ar[r]\ar[d]_{\ol g} 
& \wti{S}\ar[r] \ar[d]_{\wti g} 
& \shift Y\ar[r] \ar[d]_{\shift g^*}
& 0
\\
0\ar[r] 
& V'\ar[r]
& \wti{S'}\ar[r]  
& \shift Y'\ar[r] 
& 0
}$$
wherein each row is an exact sequence as in part~\eqref{l01a},
the morphisms $\ol g$ and $g^*$ are lifts of $g$,
and $\wti g$ is induced by a lift of $g$. \qed
\end{enumerate}
\end{lem}

We end this section with two examples.
The first one shows that, even when $\catw$
is a projective generator and an injective cogenerator for $\catx$, 
one may have $\catx\subsetneq\opg(\catw)$.

\begin{ex} \label{ex01}
Let $R$ be a commutative noetherian local ring with residue field $k$.
Let $\catw$ denote the category of finite rank free $R$-modules.
Let $\catx$ denote the category of finitely generated $R$-modules $G$
in $\catgp(R)$
with finite complexity, that is, such that the sequence of Betti numbers
$\{\beta^R_i(G)\}$ is bounded above by a polynomial in $i$.
(The category $\catx$ was studied by Gerko~\cite{gerko:ohd}.)
It is straightforward to show that $\catw\subseteq\catx\subseteq\opg(\catw)$
and that $\catw$ a projective generator and an injective cogenerator for $\catx$.
Furthermore, if $R$ is artinian and Gorenstein,
then $k\in\opg(\catw)$. If $R$ is not a complete intersection,
then $k\notin\catx$ because $k$ has infinite complexity,
so we have $\catx\subsetneq\opg(\catw)$ in this case.
\end{ex}

Our next example shows that some categories  are not perfectly suited for studying in this context.

\begin{ex} \label{ex02}
Let $R$ be a commutative noetherian ring.
An $R$-module $G$ is \emph{strongly Gorenstein projective}
if it is in $\catgp(R)$ with complete projective resolution 
that is periodic of period 1, that is, of the form
$\cdots \xra{\partial}P\xra{\partial}P\xra{\partial}P\xra{\partial}\cdots$.
These modules were introduced by Bennis and
Mahdou~\cite{bennis:sgpifm}
who prove that an $R$-module is  in $\catgp(R)$
if and only if it is a direct
summand of a strongly Gorenstein projective $R$-module.
Let $\catgp^{{s}}(R)$ denote the category of
strongly Gorenstein projective modules. 
Then we have $\catp(R)\subseteq\catgp^{{s}}(R)\subseteq\catgp(R)$,
and $\catp(R)$ is a projective generator and an injective cogenerator for $\catgp^{{s}}(R)$.

On the surface, it looks as though our results should apply to
the category $\catx=\catgp^{{s}}(R)$.
However, this category is not closed under direct summands in general
(see~\cite[(3.11)]{bennis:sgpifm})
so it is not exact and many our results do not apply. 
For instance, in Lemma~\ref{cplt01}, we can conclude that 
each strongly Gorenstein projective $R$-module $M$ admits a 
Tate $\catp(R)$-resolution $T\to W\to M$; however, we cannot conclude directly that
$\Ker(\partial^T_i)$ is strongly Gorenstein projective.
\end{ex}

\section{Foundations of Tate Cohomology}\label{sec4}

This section contains fundamental results on Tate cohomology functors,
including the proof of Theorem~\ref{thmA}. 

\begin{defn} \label{tate01}
Let $M,M',N,N'$ be objects in $\cata$ 
equipped with homomorphisms
$f\colon M\to M'$ and $g\colon N\to N'$.
If $M$ admits a Tate $\catw$-resolution 
$T\xra{\alpha}W\xra{\gamma} M$, define
the \emph{$n$th Tate cohomology group} $\watate^n(M,N)$ as
$$\watate^n(M,N)=\HH_{-n}(\ahom(T,N))$$
for each integer $n$.
If $M'$ also admits a Tate $\catw$-resolution 
$T'\xra{\alpha}W'\xra{\gamma} M'$, then 
let $\comp{f}$ be as in Lemma~\ref{comp01} and define
\begin{gather*}
\watate^n(f,N)=\HH_{-n}(\ahom(\comp{f},N))\colon\watate^n(M',N)\to\watate^n(M,N) \\
\watate^n(M,g)=\HH_{-n}(\ahom(T,g))\colon\watate^n(M,N)\to\watate^n(M,N'). 
\end{gather*}
The following \emph{comparison homomorphisms}
$$\wacomp^n(M,N)=\HH_{-n}(\Hom(\alpha,N))\colon \waext^n(M,N)\to\watate^n(M,N) $$
make the next diagram commute for each integer $n$
$$
\xymatrix{
\waext^n(M',N) \ar[rr]^{\waext^n(f,N)}\ar[d]^{\wacomp^n(M',N)}  
&& \waext^n(M,N) \ar[rr]^{\waext^n(M,g)}\ar[d]^{\wacomp^n(M,N)}  
&& \waext^n(M,N') \ar[d]^{\wacomp^n(M,N')}  \\
\watate^n(M',N) \ar[rr]^{\watate^n(f,N)} 
&& \watate^n(M,N) \ar[rr]^{\watate^n(M,g)}  
&& \watate^n(M,N'). 
}
$$
On the other hand, if $N$ admits a Tate $\catv$-coresolution 
$N\xra{\delta} V\xra{\beta} S$, define
the \emph{$n$th Tate cohomology group} $\avtate^n(M,N)$ as
$$\avtate^n(M,N)=\HH_{-n}(\ahom(M,S))$$
for each integer $n$.
If $N'$ also admits a Tate $\catv$-coresolution 
$N'\xra{\delta'} V'\xra{\beta'} S'$, then 
let $\comp{g}$ be as in Lemma~\ref{comp01} and define
\begin{gather*}
\avtate^n(f,N)=\HH_{-n}(\ahom(f,S))\colon\avtate^n(M',N)\to\avtate^n(M,N) \\
\avtate^n(M,g)=\HH_{-n}(\ahom(M,\comp{g}))\colon\avtate^n(M,N)\to\avtate^n(M,N'). 
\end{gather*}
The following \emph{comparison homomorphisms}
$$\avcomp^n(M,N)=\HH_{-n}(\Hom(M,\beta))\colon \avext^n(M,N)\to\avtate^n(M,N) $$
make the next diagram commute for each integer $n$
$$
\xymatrix{
\avext^n(M',N) \ar[rr]^{\avext^n(f,N)}\ar[d]^{\avcomp^n(M',N)}  
&& \avext^n(M,N) \ar[rr]^{\avext^n(M,g)}\ar[d]^{\avcomp^n(M,N)}  
&& \avext^n(M,N') \ar[d]^{\avcomp^n(M,N')}  \\
\avtate^n(M',N) \ar[rr]^{\avtate^n(f,N)} 
&& \avtate^n(M,N) \ar[rr]^{\avtate^n(M,g)}  
&& \avtate^n(M,N'). 
}
$$
\end{defn}

\begin{fact} \label{tate02'}
Let $R$ be a commutative ring, and assume that $\catw$  and $\catv$ are subcategories of
$\cata=\catm(R)$. Let $M,M',N,N'$ be $R$-modules
equipped with $R$-module homomorphisms
$f\colon M\to M'$ and $g\colon N\to N'$.
If $M$ admits a Tate $\catw$-resolution, then
each group $\watate^n(M,N)$
is an $R$-module, and the comparison
maps $\wacomp^n(M,N)$
are $R$-module homomorphisms.
If $M'$ also admits a Tate $\catw$-resolution, then 
the maps
$\watate^n(f,N)$ and $\watate^n(M,g)$
are $R$-module homomorphisms.
Similar comments hold for $\avtate^n$
and $\avcomp^n(M,N)$.
\end{fact}

\begin{fact} \label{tate02}
Lemma~\ref{comp01} parts~\eqref{comp01item1} and~\eqref{comp01item2} show that 
\begin{align*}
\watate^n&\colon\cpltrescatw\times\cata\to \catab 
&\text{and}&&
\avtate^n&\colon\cata\times \cpltcorescatv\to \catab
\end{align*}
are well-defined bifunctors and 
that 
\begin{align*}
\wacomp^n&\colon\waext^n|_{\cpltrescatw\times\cata}\to\watate^n &
\avcomp^n&\colon\avext^n|_{\cata\times\cpltcorescatv}\to\avtate^n
\end{align*}
are natural transformations,  independent of resolutions and liftings.
\end{fact}

\begin{notation} \label{notn2}
Let $R$ be a commutative ring, and let $C$ be a
semidualizing $R$-module.  
We  abbreviate as follows:
\begin{align*}
\smash{\comp{\ext}}_{\catpc}
&=\smash{\comp{\ext}}_{\catpc(R) \catm(R)}
&\smash{\comp{\ext}}_{\catic}
&=\smash{\comp{\ext}}_{\catm(R) \catic(R)}\\
\ext_{\catpc}
&=\ext_{\catpc(R) \catm(R)}
&\ext_{\catic}
&=\ext_{\catm(R) \catic(R)}\\
\ext_{\opg(\catpc)}
&=\ext_{\opg(\catpc(R)) \catm(R)}
&\ext_{\opg(\catic)}
&=\ext_{\catm(R) \opg(\catic(R))}. 
\end{align*}
\end{notation}

The next result show that objects with finite homological dimensions
have vanishing Tate cohomology, as in~\cite[(5.2)]{avramov:aratc}. 
See Theorems~\ref{tatewhat01} and~\ref{tatewhat02}
for converses.

\begin{prop} \label{vanish01}
Let $M$ and $N$ be objects in $\cata$, and assume $\catw\perp\catw$ and $\catv\perp\catv$.
\begin{enumerate}[\quad\rm(a)]
\item \label{vanish01item1}
If $\wpd(M)<\infty$, then  $\watate^n(M,-)=0$ and
$\watate^n(-,M)=0$ for all $n$.
\item \label{vanish01item3}
If $\vid(N)<\infty$,  then   $\avtate^n(-,N)=0$ 
and $\avtate^n(N,-)=0$ for all $n$.
\end{enumerate}
\end{prop}

\begin{proof}
We prove part~\eqref{vanish01item1};
the proof of~\eqref{vanish01item3} is dual.
Assume that $\wpd(M)<\infty$.
The vanishing $\watate(M,-)=0$  follows from Remark~\ref{cplt03},
since we have a Tate $\catw$-resolution of $M$
of the form $0\to W\to M$.
The vanishing $\watate(-,M)=0$
follows from the last part of Fact~\ref{gtriv01} since,
for each complete $\catw$-resolution $T'$, the complex
$\ahom(T',M)$ is exact.
\end{proof}

Our next results provide long exact sequences for Tate cohomology.
They are proved like~\cite[(5.4),(5.6)]{avramov:aratc},
using Lemma~\ref{horseshoe02}.

\begin{lem} \label{les01}
Let $M$ be an object in $\cpltrescatw$,
and let $N$ be an object in $\cpltcorescatv$.
Consider an exact sequence in $\cata$ 
$$\mathbb{L}=\quad 0\to L'\xra{f'}L\xra{f}L''\to 0.$$
\begin{enumerate}[\quad\rm(a)]
\item \label{les01item1}
If the sequence $\mathbb{L}$ is $\ahom(\catw,-)$-exact,
then there is a long exact sequence
\begin{align*}
\cdots \to
& \watate^n(M,L') \xra{\watate^n(M,f')}
  \watate^n(M,L) \xra{\watate^n(M,f)}
    \\
& \watate^n(M,L'')\xra{\comp{\eth}^n_{\catw\cata}(M,\mathbb{L})}
 \watate^{n+1}(M,L') \xra{\watate^{n+1}(M,f')}
 \cdots 
\end{align*}
that is natural in $M$ and $\mathbb{L}$, and is compatible with the long exact sequence in relative
cohomology via the comparison maps $\wacomp^n$ from~\ref{tate01}.
\item \label{les01item2}
If the sequence $\mathbb{L}$  is $\ahom(-,\catv)$-exact,
then there is a long exact sequence
\begin{align*}
\cdots \to
& \avtate^n(L'',N) \xra{\avtate^n(f,N)}
 \avtate^n(L,N) \xra{\avtate^n(f',N)}
  \\
& \avtate^n(L',N)\xra{\comp{\eth}^n_{\cata\catv}(\mathbb{L},N)}
 \avtate^{n+1}(L'',N) \xra{\avtate^{n+1}(f,N)}
 \cdots 
\end{align*}
that is natural in $N$ and $\mathbb{L}$, and is compatible with the long exact sequence in relative
cohomology via the comparison maps $\avcomp^n$ from~\ref{tate01}. \qed
\end{enumerate}
\end{lem}

\begin{lem} \label{les02}
Let $M$ and $N$ be objects in $\cata$, and assume that $\catw\perp\catw$ and $\catv\perp\catv$.
Consider an exact sequence in $\cata$ 
$$\mathbb{L}=\quad 0\to L'\xra{f'}L\xra{f}L''\to 0.$$
\begin{enumerate}[\quad\rm(a)]
\item \label{les02item2}
Assume that $\catw$ is closed under kernels of epimorphisms, the objects
$L,L',L''$ are in $\cpltrescatw$,
and the sequence $\mathbb{L}$ is $\ahom(\catw,-)$-exact.
Then there is a long exact sequence
\begin{align*}
\cdots \to
& \watate^n(L'',N) \xra{\watate^n(f,N)}
 \watate^n(L,N) \xra{\watate^n(f',N)}
 \\
&  \watate^n(L',N)\xra{\comp{\eth}^n_{\catw\cata}(\mathbb{L},N)}
 \watate^{n+1}(L'',N) \xra{\watate^{n+1}(f,N)}
 \cdots 
\end{align*}
that is natural in $N$ and $\mathbb{L}$, and is compatible with the long exact sequence in relative
cohomology via the comparison maps $\wacomp^n$.
\item \label{les02item1}
Assume that $\catv$ is closed under cokernels of monomorphisms, the objects
$L,L',L''$ are in $\cpltcorescatv $,
and the sequence $\mathbb{L}$ is $\ahom(-,\catv)$-exact.
Then there is a long exact sequence
\begin{align*}
\cdots \to
& \avtate^n(M,L') \xra{\avtate^n(M,f')}
 \avtate^n(M,L) \xra{\avtate^n(M,f)}
  \\
& \avtate^n(M,L'')\xra{\comp{\eth}^n_{\cata\catv}(M,\mathbb{L})}
 \avtate^{n+1}(M,L') \xra{\avtate^{n+1}(M,f')}
 \cdots 
\end{align*}
that is natural in $M$ and $\mathbb{L}$, and is compatible with the long exact sequence in relative
cohomology via the comparison maps $\avcomp^n$. \qed
\end{enumerate}
\end{lem}

The next two lemmas  allow us to dimension-shift with Tate cohomology.
They have similar proofs, as do the other natural invariants.

\begin{lem} \label{dimshft03}
Assume that $\catw\perp\catw$, and consider an exact sequence in $\cata$
$$\mathbb{L}=\quad 0\to L'\to L\to L''\to 0$$
that is $\ahom(\catw,-)$ exact  and such that $L\in\finrescatw$.
\begin{enumerate}[\quad\rm(a)]
\item \label{dimshft03item1}
The natural transformation
$\comp{\eth}^n(-,\mathbb{L})\colon \watate^n(-,L'')\xra{\cong}\watate^{n+1}(-,L')$
is an isomorphism of functors 
for each $n\in\mathbb{Z}$.
\item \label{dimshft03item2}
If $\catw$ is closed under
kernels of epimorphisms and $L',L''\in\cpltrescatw$, then the natural transformation
$\comp{\eth}^n(\mathbb{L},-)\colon \watate^n(L',-)\xra{\cong}\watate^{n+1}(L'',-)$
is an isomorphism of functors  for each $n\in\mathbb{Z}$. 
\end{enumerate}
\end{lem}

\begin{proof}
\eqref{dimshft03item1}
Use the long exact sequence from Lemma~\ref{les01}\eqref{les01item1}
with the vanishing $\watate^n(-,L)=0$ from Propositon~\ref{vanish01}\eqref{vanish01item1}.

\eqref{dimshft03item2}
Our hypotheses guarantee that 
the functors and tranfsormation under consideration are defined.  
Now use the long exact sequence from Lemma~\ref{les02}\eqref{les02item1}
with the vanishing $\watate^n(L,-)=0$ from Propositon~\ref{vanish01}\eqref{vanish01item1}.
\end{proof}

\begin{lem} \label{dimshft04}
Assume that $\catv\perp\catv$, and consider an exact sequence in $\cata$
$$\mathbb{L}=\quad 0\to L'\to L\to L''\to 0$$
that is $\ahom(-,\catv)$ exact  and such that $L\in\fincorescatv$.
\begin{enumerate}[\quad\rm(a)]
\item \label{dimshft04item1}
The natural transformation
$\comp{\eth}^n(\mathbb{L},-)\colon \avtate^n(L',-)\xra{\cong}\avtate^{n+1}(L'',-)$
is an isomorphism of functors 
for each $n\in\mathbb{Z}$.
\item \label{dimshft04item2}
If $\catv$ is closed under
cokernels of monomorphisms and $L',L''\in \cpltcorescatv$, then the natural transformation
$\comp{\eth}^n(-,\mathbb{L})\colon \avtate^n(-,L'')\xra{\cong}\avtate^{n+1}(-,L')$
is an isomorphism of functors  for each $n\in\mathbb{Z}$. \qed
\end{enumerate}
\end{lem}

Next, we connect relative and Tate cohomology via a long exact sequence.

\begin{thm} \label{mixed01}
Assume that $\catx$ is exact and closed under kernels of epimorphisms.
Assume that $\catw$ is closed under direct summands and is both an injective cogenerator 
and a projective generator for $\catx$.
Fix objects $M\in\finrescatx$ and $N\in\cata$, and set $d=\xpd(M)$.
There is a long exact sequence
\begin{align*}
0 \to
& \xaext^1(M,N) \xra{\xwacomp^1(M,N)}
 \waext^1(M,N) \xra{\wacomp^1(M,N)}
 \watate^1(M,N) \to 
  \\
\to
& \xaext^2(M,N) \xra{\xwacomp^n(M,N)}
 \waext^2(M,N) \xra{\wacomp^n(M,N)}
 \watate^2(M,N) \to 
  \\
\cdots \to
& \xaext^d(M,N) \xra{\xwacomp^d(M,N)}
 \waext^d(M,N) \xra{\wacomp^d(M,N)}
 \watate^d(M,N) \to 
0
\end{align*}
that is natural in $M$ and $N$,
and the next maps are isomorphisms  for each $n>d$
$$ \wacomp^n(M,N)\colon \waext^n(M,N) \xra{\cong}
 \watate^n(M,N). $$
\end{thm}

\begin{proof}
By Lemma~\ref{cplt01}\eqref{cplt01item1} there is a
Tate $\catw$-resolution
$T\to W\to M$ such that $\alpha_n$ is a split surjection for each $n$.
Lemma~\ref{l01}\eqref{l01a} yields a degreewise split exact sequence of complexes
\begin{equation} \label{exact23}
0\to \shift^{-1}X\to T'\to W\to 0
\end{equation}
wherein
$X$ is a bounded strict $\catw\catx$-resolution of $M$ and
$T'_{\geq 0}\cong T_{\geq 0}$.
In particular, there are isomorphisms for each $n\geq 1$
\begin{gather*}
\xaext^n(M,N)\cong \HH_{-n}(\ahom(X,N)) \qquad
\waext^n(M,N)\cong \HH_{-n} (\ahom(W,N))\\
\watate^n(M,N)\cong \HH_{-n} (\ahom(T',N)).
\end{gather*}
Recall that $\xaext^n(M,N)=0$ for $n>d$.
Apply the functor $\ahom(-,N)$ to the sequence~\eqref{exact23} 
and take the induced long exact sequence
to obtain the desired long exact sequence and the isomorphisms.

To show that the long exact sequence is natural in $N$, 
let $g\colon N\to N'$ be a morphism in $\cata$. Apply $\ahom(-,g)$ to the sequence~\eqref{exact23}
to obtain the next commutative diagram
$$\xymatrix{
0\ar[r] 
& \ahom(W,N)\ar[r]\ar[d]
& \ahom(\wti{T},N)\ar[r] \ar[d]
& \ahom(\shift^{-1}X,N)\ar[r] \ar[d]
& 0
\\
0\ar[r] 
& \ahom(W,N')\ar[r]
& \ahom(\wti{T},N')\ar[r]  
& \ahom(\shift^{-1}X,N')\ar[r] 
& 0
}$$
which induces a commutative diagram of long exact sequences, as desired.

To show that the long exact sequence is natural in $M$, 
let $f\colon M\to M'$ be a morphism in $\cata$. 
Apply $\ahom(-,N)$ to the diagram from Lemma~\ref{l01}\eqref{l01b}
to obtain the next commutative diagram
$$\xymatrix{
0\ar[r] 
& \ahom(W',N)\ar[r]\ar[d]
& \ahom(\wti{T'},N)\ar[r] \ar[d]
& \ahom(\shift^{-1}X',N)\ar[r] \ar[d]
& 0
\\
0\ar[r] 
& \ahom(W,N)\ar[r]
& \ahom(\wti{T},N)\ar[r]  
& \ahom(\shift^{-1}X,N)\ar[r] 
& 0
}$$
which induces the desired commutative diagram of long exact sequences.
\end{proof}

\begin{para} \label{prfA}
\emph{Proof of Theorem~\ref{thmA}.}
Fact~\ref{disc01} shows that hypotheses of Theorem~\ref{mixed01}
are satisfied by $\catw=\catpc(R)$ and $\catx=\opg(\catpc(R))$.
\qed
\end{para}

The proofs of the next results are dual to those of Theorem~\ref{mixed01}
and Theorem~\ref{thmA}.

\begin{thm} \label{mixed02}
Assume that $\caty$ is exact and closed under cokernels of monomorphisms.
Assume that $\catv$  is closed under direct summands and is  an injective cogenerator 
and a projective generator for $\caty$.
Fix objects $M\in\cata$ and $N\in\fincorescaty$, and set $d=\yid(N)$.
There is a long exact sequence
\begin{align*}
0 \to
& \ayext^1(M,N) \xra{\ayvcomp^1(M,N)}
 \avext^1(M,N) \xra{\avcomp^1(M,N)}
 \avtate^1(M,N) \to 
  \\
\to
& \ayext^2(M,N) \xra{\ayvcomp^n(M,N)}
 \avext^2(M,N) \xra{\avcomp^n(M,N)}
 \avtate^2(M,N) \to 
  \\
\cdots \to
& \ayext^d(M,N) \xra{\ayvcomp^d(M,N)}
 \avext^d(M,N) \xra{\avcomp^d(M,N)}
 \avtate^d(M,N) \to 
0
\end{align*}
that is natural in $M$ and $N$,
and the next maps are isomorphisms for each $n>d$
\begin{xxalignat}{3}
  &{\hphantom{\square}}& \avcomp^n(M,N)
  &\colon \avext^n(M,N) \xra{\cong}
 \avtate^n(M,N).  &&\qed
\end{xxalignat}
\end{thm}

\begin{cor} \label{thmA'}
Let $R$ be a commutative ring, and let $C$ be a semidualizing $R$-module.
Let  $M$ and $N$ be $R$-modules, and assume that
$d=\opg(\catic)\text{-}\id_R(N)<\infty$.
There is a long exact sequence that is natural in $M$ and $N$
\begin{align*}
0 \to
& \ext_{\opg(\catic)}^1(M,N) \to
 \ext_{\catic}^1(M,N) \to
\smash{\comp{\ext}}_{\catic}^1(M,N) \to 
  \\
\to
& \ext_{\opg(\catic)}^2(M,N) \to
 \ext_{\catic}^2(M,N) \to
\smash{\comp{\ext}}_{\catic}^2(M,N) \to 
  \\
\cdots \to
& \ext_{\opg(\catic)}^d(M,N) \to
 \ext_{\catic}^d(M,N) \to
 \smash{\comp{\ext}}_{\catic}^d(M,N) \to 
0
\end{align*}
and there are isomorphisms  
$\ext_{\catic}^n(M,N) \xra{\cong}
\smash{\comp{\ext}}_{\catic}^n(M,N)$
for each $n>d$. \qed
\end{cor}

\section{Vanishing of Tate Cohomology}\label{sec10}

This section focuses on the interplay between finiteness of homological
dimensions and vanishing of Tate cohomology. It contains the proof
of Theorem~\ref{thmB}.
We begin with
a result that compares to~\cite[(5.9)]{avramov:aratc}, though 
the proof is
different.

\begin{lem} \label{tateperp01}
Assume that $\catw$ is closed under direct summands, and let
$M\in\opg(\catw)$. If $\watate^0(M,M)=0$
or $\awtate^0(M,M)=0$, then $M$ is in $\catw$.
\end{lem}

\begin{proof}
We prove the case where $\watate^0(M,M)=0$;
the proof of the other case is dual.
From Remark~\ref{cplt03} there is a Tate $\catw$-resolution
$T\xra\alpha W\xra\gamma M$ such that $\alpha_n$ is an isomorphism
for all $n\geq 0$.
This induces the second and third isomorphisms in the following sequence
$$\im(\partial^T_0)\xla[\cong]{\sigma} \coker(\partial^T_1)\xra[\cong]{\ol{\alpha_0}}
\coker(\partial^W_1)\xra[\cong]{\ol{\gamma_0}}
M.
$$
The 
first isomorphism comes from the exactness of $T$.
It is straightforward to show that the left-most rectangle in
the following diagram commutes
\begin{equation*} 
\xymatrix{
W_0\ar[d]_{\gamma_0}&T_0\ar[l]_-{\alpha_0}^-\cong \ar[d]_{\pi}\ar[r]^-{\partial^T_0} &T_{-1} \\
M & \coker(\partial^T_1) \ar[l]_-{\ol{\gamma_0}\,\ol{\alpha_0}}^-\cong\ar[r]^-\sigma_-\cong
&\im(\partial^T_0).\ar@{^(->}[u]_-\epsilon 
}
\end{equation*}
Here, the morphisms $\pi$ and $\epsilon$ are the natural surjection and injection, respectively,
and it follows that the
right-most rectangle  also commutes.
This diagram provides a monomorphism
$f=\epsilon\sigma(\ol{\gamma_0}\,\ol{\alpha_0})^{-1}\colon M\hookrightarrow T_{-1}$ such that 
\begin{equation} \label{compose02}
f\gamma_0\alpha_0=\partial^T_0.
\end{equation}
The vanishing hypothesis 
$$0=\watate^0(M,M)=\HH_0(\ahom(T,M))$$
implies that every chain map $T\to M$ is null-homotopic.
In particular, the chain map $T\xra{\gamma\alpha} M$ is null-homotopic
with homotopy $s$ as in the next diagram
$$\xymatrix{
\cdots\ar[rr]^{\partial^T_2} 
&& T_1\ar[rr]^{\partial^T_1}\ar[d]^{\gamma_{1}\alpha_{1}} 
&& T_0\ar[rr]^{\partial^T_0}\ar[d]^{\gamma_{0}\alpha_{0}} \ar[lld]^{s_0=0}
&& T_{-1}\ar[rr]^{\partial^T_{-1}}\ar[d]^{\gamma_{-1}\alpha_{-1}} \ar[lld]^{s_{-1}}
&& \cdots \\
\cdots\ar[rr] && 0 \ar[rr] && M \ar[rr] && 0 \ar[rr] && \cdots.
}
$$
This yields a morphism $s_{-1}\colon T_{-1}\to M$ such that 
\begin{equation} \label{compose01}
\gamma_0\alpha_0=s_{-1}\partial^T_0.
\end{equation}
Combine~\eqref{compose01} and~\eqref{compose02} to obtain the following sequence
$$s_{-1}f\gamma_0\alpha_0=s_{-1}\partial^T_0=\gamma_0\alpha_0=\id_M\gamma_0\alpha_0$$
and use the fact that $\gamma_0\alpha_0$ is surjective to conclude that $s_{-1}f=\id_M$.
Thus, the morphism $f\colon M\to T_{-1}$ is a split monomorphism.
Since $\catw$ is closed under direct summands and $T_{-1}$ is in $\catw$,
it follows that $M$ is in $\catw$, as desired.
\end{proof}

The next result contains a partial converse 
Proposition~\ref{vanish01}\eqref{vanish01item1}, as in~\cite[(5.9)]{avramov:aratc}.

\begin{thm} \label{tatewhat01}
Assume that $\catw$ is exact and 
closed under kernels of epimorphisms
and that $\catw\perp\catw$.
For an object $M\in\finrescat{G(W)}$, the next conditions are equivalent:
\begin{enumerate}[\quad\rm(i)]
\item \label{tatewhat01item1}
$\wpd(M)<\infty$;
\item \label{tatewhat01item3}
$\watate^n(-,M)=0$  for each (equivalently, for some) $n\in\mathbb{Z}$;
\item \label{tatewhat01item2}
$\watate^n(M,-)=0$  for each (equivalently, for some) $n\in\mathbb{Z}$;
and
\item \label{tatewhat01item6}
$\watate^0(M,M)=0$.
\end{enumerate}
\end{thm}

\begin{proof}
Fact~\ref{sgcsummary} yields a $\catw\opg(\catw)$-hull
\begin{equation} \label{exact21} 
0\to M\to K \to M^{(-1)}\to 0
\end{equation}
that is, an exact sequence  with $K\in\finrescatw$ and $M^{(-1)}\in\opg(\catw)$;
see Definition~\ref{ab01}. 
Fact~\ref{sgcsummary} implies that $\catw\perp M$,
so the sequence~\eqref{exact21} is $\ahom(\catw,-)$-exact.
From the assumption 
$\catw\perp\catw$, we conclude that $\catw\perp\finrescatw$.
In particular, we have
$\catw\perp K$, and a standard  argument 
implies that $\catw\perp M^{(-1)}$.

Fact~\ref{sgcsummary} shows that 
Lemma~\ref{cplt01}\eqref{cplt01item1} applies to the category $\catx=\opg(\catw)$.
So, the object $M\in \finrescat{G(W)}$ 
admits a proper $\catw$-resolution $W\xra{\gamma} M$.

The implication \eqref{tatewhat01item1}$\implies$\eqref{tatewhat01item3} 
follows from
Proposition~\ref{vanish01}\eqref{vanish01item1}.

\eqref{tatewhat01item3}$\implies$\eqref{tatewhat01item6}.
Assume that $\watate^n(-,M)=0$ for  some $n\in\mathbb{Z}$.
If $n=0$, then condition~\eqref{tatewhat01item6} follows immediately.

Assume next  that $n<0$. 
Set $M^{(0)}=M$ and $M^{(i)}=\im(\partial^W_i)$ for each $i\geq 1$.
The next exact sequences
are  $\ahom(\catw,-)$-exact because $W$ is a proper resolution
\begin{equation} \label{exact20} 
0\to M^{(i)}\to W_{i-1}\to M^{(i-1)}\to 0.
\end{equation}
Since $M^{(0)},W_{i}\in\cpltrescatw$,
induction on $i$ implies that each $M^{(i)}$ is in $\cpltrescatw$ by
Corollary~\ref{cplt05}\eqref{cplt05item1}.
Repeated application of Lemma~\ref{dimshft03}\eqref{dimshft03item2}
yields the  isomorphisms in the following sequence
$$\watate^0(M,M)=\watate^0(M^{(0)},M)\cong 
\watate^{n}(M^{(-n)},M)=0$$
while the vanishing is by hypothesis.

Assume next  $n>0$. 
The object $M^{(-1)}$ from~\eqref{exact21} is in $\opg(\catw)$.
For $i\leq -2$ use the complete $\catw$-resolution of $M^{(-1)}$
to construct exact sequences
\begin{equation*}  
0\to M^{(i+1)}\to W_{i}\to M^{(i)}\to 0
\end{equation*}
with $W_i\in\catw$ and
$M^{(i)}\in\opg(\catw)$.
Since the complete $\catw$-resolution of $M^{(-1)}$
is $\ahom(\catw,-)$-exact, the same is true of each of these sequences.
A standard argument shows that
$\catw\perp M^{(i)}$ for each $i\leq 2$.
Repeated application of Lemma~\ref{dimshft03}\eqref{dimshft03item2}
yields the  isomorphisms in the following sequence
$$\watate^0(M,M)\cong \watate^{1}(M^{(-1)},M)\cong \cdots
\cong \watate^{n}(M^{(-n)},M)=0$$
while the vanishing is by hypothesis. 

The  implications
\eqref{tatewhat01item1}$\implies$\eqref{tatewhat01item2}$\implies$\eqref{tatewhat01item6}
are verified similarly.

\eqref{tatewhat01item6}$\implies$\eqref{tatewhat01item1}.
Assume $\watate^0(M,M)=0$ and again consider the $\catw\catx$-hull~\eqref{exact21}.
The isomorphisms in the following sequence are from Lemma~\ref{dimshft03},
$$\watate^0(M^{(-1)},M^{(-1)})\cong\watate^{-1}(M,M^{(-1)})\cong\watate^{0}(M,M)=0$$
while the vanishing is by hypothesis. 
Since $M^{(-1)}\in\opg(\catw)$, Lemma~\ref{tateperp01}
implies that $M^{(-1)}\in\catw$.  Since $K$ is in $\finrescatw$, the exact sequence~\eqref{exact21}
implies that $M\in\finrescatw$, using~\cite[(3.5)]{auslander:htmcma}.
\end{proof}

\begin{para} \label{prfB}
\emph{Proof of Theorem~\ref{thmB}.}
Theorem~\ref{tatewhat01}
applies to $\catw=\catpc(R)$ by Fact~\ref{disc01}.
\qed
\end{para}

The proofs of the next results are dual to those of Theorem~\ref{tatewhat01}
and Theorem~\ref{thmB}.

\begin{thm} \label{tatewhat02}
Assume that $\catv$ is exact and 
closed under cokernels of monomorphisms
and that $\catv\perp\catv$.
For each $M\in\fincorescat{G(V)}$, the following are equivalent:
\begin{enumerate}[\quad\rm(i)]
\item \label{tatewhat02item1}
$\vid(M)<\infty$;
\item \label{tatewhat02item3}
$\avtate^n(-,M)=0$  for each (equivalently, for some) $n\in\mathbb{Z}$;
\item \label{tatewhat02item2}
$\avtate^n(M,-)=0$   for each (equivalently, for some) $n\in\mathbb{Z}$;
and
\item \label{tatewhat02item6}
$\avtate^0(M,M)=0$. \qed
\end{enumerate}
\end{thm}

\begin{cor} \label{thmB'}
Let $R$ be a commutative ring, and let $C$ be a semidualizing $R$-module.
For an $R$-module $M$ with
$\opg(\catic)\text{-}\id_R(M)<\infty$,
the following are equivalent:
\begin{enumerate}[\quad\rm(i)]
\item \label{thmB'1}
$\catic\text{-}\id_R(M)<\infty$;
\item \label{thmB'3}
$\smash{\comp{\ext}}_{\catic}^n(M,-)=0$ 
for each (equivalently, for some) $n\in\mathbb{Z}$;
\item \label{thmB'2}
$\smash{\comp{\ext}}_{\catic}^n(-,M)=0$  for each (equivalently, for some) $n\in\mathbb{Z}$;
and
\item \label{thmB'6}
$\smash{\comp{\ext}}_{\catic}^0(M,M)=0$.\qed
\end{enumerate}
\end{cor}

The next two results compare to~\cite[(7.2)]{avramov:aratc} and~\cite[(4.8)]{sather:gcac}.

\begin{cor} \label{isos}
Assume that $\catx$ is exact and closed under kernels of epimorphisms.
Assume that $\catw$ is closed under direct summands and kernels of epimorphisms.
Assume that $\catw$ is both an injective cogenerator 
and a projective generator for $\catx$.
Let $M$ be an object of $\cata$ with $d=\xpd(M)<\infty$. The next conditions are equivalent:
\begin{enumerate}[\quad\rm(i)]
\item \label{isositem1i}
$\wpd(M)<\infty$;
\item \label{isositem1ii}
The natural transformation $\xwacomp^i(M,-)\colon\xaext^i(M,-)\xra{\cong}\waext^i(M,-)$ 
is an isomorphism 
for each $i$; and
\item \label{isositem1iii}
The natural transformation
$\xwacomp^i(M,-)\colon\xaext^i(M,-)\xra{\cong}\waext^i(M,-)$ is an isomorphism  
either for two successive values of  $i$ with $1\leq i<d$ or for a single value of $i$ with $i\geq d$.
\end{enumerate}
\end{cor}

\begin{proof}
The implication~\eqref{isositem1i}$\implies$\eqref{isositem1ii} is in~\cite[(4.8)]{sather:gcac}, and
\eqref{isositem1ii}$\implies$\eqref{isositem1iii} is trivial.  

For \eqref{isositem1iii}$\implies$\eqref{isositem1i}, we consider three cases.

Case 1: The natural transformations $\xwacomp^i(M,-)$ and $\xwacomp^{i+1}(M,-)$ are 
isomorphisms  
where $1\leq i<d-1$.
In this case, use the long exact
sequence in Theorem~\ref{mixed01} to conclude that
$\watate^i(M,-)=0$.  The conclusion
$\wpd(M)<\infty$ then follows from Theorem~\ref{tatewhat01}.

Case 2: The natural transformation $\xwacomp^d(M,-)$ is an isomorphism.
As in Case 1, we conclude that $\watate^d(M,-)=0$ and hence
$\wpd(M)<\infty$.

Case 3: The natural transformation $\xwacomp^i(M,-)$ is an isomorphism
for some $i>d$. Our assumption yields the isomorphism in the next sequence
$$\waext^i(M,-)\cong\xaext^i(M,-)=0$$
while the vanishing is from~\cite[(4.5.b)]{sather:gcac} since $i>d=\xpd(M)$.
From~\cite[(4.5.a)]{sather:gcac} we conclude that
$\wpd(M)<i<\infty$.
\end{proof}

Our next result augments the previous one in the special case $\catx=\opg(\catw)$.

\begin{cor} \label{isos'}
Assume that $\catw\perp\catw$ and that
$\catw$ is closed under direct summands and kernels of epimorphisms.
Let $M$ be an object of $\cata$ with $d=\gwpd(M)<\infty$. The following conditions are equivalent:
\begin{enumerate}[\quad\rm(i)]
\item \label{isos'item1i}
$\wpd(M)<\infty$;
\item \label{isos'item1iv}
The transformation
$\gwwacomp^i(-,M)\colon\gwaext^i(-,M)\xra{\cong}\waext^i(-,M)$
is an isomorphism on $\cpltrescatw$
for each $i$; and
\item \label{isos'item1v}
The transformation $\gwwacomp^i(-,M)\colon\gwaext^i(-,M)\xra{\cong}\waext^i(-,M)$
is an isomorphism on $\cpltrescatw $
either for a single value of $i$ with $i\geq d$
or for two successive values of  $i$ with $1\leq i<d$.
\end{enumerate}
\end{cor}

\begin{proof}
First note that $\cpltrescatw\subseteq\proprescatw$ by Remark~\ref{cplt02}.
Furthermore, we have
$\cpltrescatw=\finrescat{G(W)}
\subseteq\proprescat{G(W)}$
by Theorem~\ref{cplt05}\eqref{cplt05item1}
and~\cite[(3.3.b)]{sather:gcac}.
The implication \eqref{isos'item1i}$\implies$\eqref{isos'item1iv} now follows from~\cite[(4.10)]{sather:gcac}.
The implication~\eqref{isos'item1iv}$\implies$\eqref{isos'item1v} is trivial,
and~\eqref{isos'item1v}$\implies$\eqref{isos'item1i}
follows as in the proof of Corollary~\ref{isos}.
\end{proof}

The proofs of the last two results of this section are dual to the previous two.

\begin{cor} \label{isos''}
Assume that $\caty$ is exact and closed under cokernels of monomorphisms.
Assume that $\catv$ is closed under direct summands and cokernels of monomorphisms.
Assume that $\catv$ is both an injective cogenerator 
and a projective generator for $\caty$.
Let $N$ be an object of $\cata$ with $d=\yid(N)<\infty$. The following conditions are equivalent:
\begin{enumerate}[\quad\rm(i)]
\item \label{isositem2i}
$\vid(N)<\infty$;
\item \label{isositem2ii}
The natural transformation
$\ayvcomp^i(-,N)\colon\ayext^i(-,N)\cong\avext^i(-,N)$
is an isomorphism 
for each $i$; and
\item \label{isositem2iii}
The natural transformation
$\ayvcomp^i(-,N)\colon\ayext^i(-,N)\cong\avext^i(-,N)$
is an isomorphism either 
for a single value of $i$ with $i\geq d$
or for two successive values of  $i$ with $1\leq i<d$. \qed
\end{enumerate}
\end{cor}

\begin{cor} \label{isos'''}
Assume that $\catv\perp\catv$ and that
$\catv$ is closed under direct summands and cokernels of monomorphisms.
Let $M$ be an object of $\cata$ with $d=\gvid(M)<\infty$. The following conditions are equivalent:
\begin{enumerate}[\quad\rm(i)]
\item \label{isos''''item1i}
$\vid(M)<\infty$;
\item \label{isos'''item1iv}
The  transformation $\agvvcomp^i(N,-)\colon\agvext^i(N,-)\cong\avext^i(N,-)$ 
is an isomorphism on $\cpltcorescatv$
for each $i$; and
\item \label{isos'''item1v}
The transformation
$\agvvcomp^i(N,-)\colon\agvext^i(N,-)\cong\avext^i(N,-)$
is an isomorphism on $\cpltcorescatv $
either 
for a single value of $i$ with $i\geq d$
or for two successive values of  $i$ with $1\leq i<d$. \qed
\end{enumerate}
\end{cor}

\section{Balance for Tate Cohomology}\label{sec6}

We begin this section with its main result,
which implies Theorem~\ref{thmC}; see~\eqref{prfC}.

\begin{thm} \label{tatebalance01}
Assume that $\catw\perp \catw$ and $\catv\perp\catv$
and $\opg(\catw)\perp\catv$ and $\catw\perp\opg(\catv)$.
Assume that $\catw$ is closed under kernels of epimorphisms and direct summands and
that $\catv$ is closed under cokernels of monomorphisms and direct summands. 
Assume also that $\waext^{\geq 1}(\finrescatw,\catv)=0=\avext^{\geq 1}(\catw,\fincorescatv)$.
For all $M\in\finrescat{G(W)}$ and all $N\in\fincorescat{G(V)}$ and all $n\geq 1$,
we have
$$\watate^{n}(M, N)\cong\avtate^{n}(M, N).$$
If, in addition, we have $\finrescatw=\fincorescatv$, then this isomorphism
holds for all $n\in\mathbb Z$.
\end{thm}

\begin{proof}
We begin by noting that~\cite[(6.4)]{sather:gcac} implies that
$\catw\perp\fincorescat{G(V)}$ and $\finrescat{G(W)}\perp\catv$.
Theorem~\ref{cplt05}\eqref{cplt05item1} yields
a Tate $\catw$-resolution
$T\xra\alpha W\to M$ such that $\alpha_n$ is a split surjection for each $n\in\mathbb Z$.
Lemma~\ref{l01}\eqref{l01a} provides a degreewise split exact sequence of complexes
\begin{equation} \label{exact24}
0\to \shift^{-1}X\to \wti{T}\to W\to 0
\end{equation}
wherein
$X$ is a bounded strict $\catw\opg(\catw)$-resolution of $M$, 
$\wti{T}$ is exact,
$\wti{T}_n=0$ for each $n<-1$, $\wti{T}_{-1}$ is in $\opg(\catw)$,
$\wti{T}_n$ is in $\catw$ for each $n\geq 0$, and
$\wti{T}_{\geq 0}\cong T_{\geq 0}$.
In particular, 
there are isomorphisms for each $n\geq 1$
\begin{gather}
\watate^n(M,-)\cong \HH_{-n} (\ahom(\wti{T},-)). \label{iso03}
\end{gather}
Similarly, 
let 
$N\xra\delta V\xra\beta L$ be a Tate $\catv$-coresolution
such that each $\beta_n$ is a split monomorphism, and
consider a degreewise split exact sequence of complexes
\begin{equation} \label{exact25}
0\to V\to \wti{S}\xra{\eta} \shift Y\to 0
\end{equation}
wherein
$Y$ is a bounded strict $\opg(\catv)\catv$-coresolution,
$\wti{S}$ is exact,
$\wti{S}_n=0$ for each $n>1$, $\wti{S}_{1}$ is in $\opg(\catv)$,
$\wti{S}_n$ is in $\catv$ for each $n\leq 0$, and
$\wti{S}_{\leq 0}\cong S_{\leq 0}$.
In particular, there are isomorphisms 
for each $n\geq 1$
\begin{gather}
\awtate^n(-,N)\cong \HH_{-n} (\ahom(-,\wti{S})). \label{iso04}
\end{gather}

The proof will be complete in the case $n\geq 1$ once we verify the 
quasiisomorphisms in the following sequence wherein 
the isomorphism in the middle is standard
\begin{equation} \label{quiso01}
\ahom(\wti{T},N)\simeq\ahom(\wti{T},\shift^{-1}\wti{S})\cong\ahom(\shift \wti{T},\wti{S})\simeq\ahom(M,\wti{S}).
\end{equation}
Indeed, this provides the second isomorphism in the following
sequence 
$$\watate^{n}(M, N)
\cong\HH_{-n} (\ahom(\wti{T},N))
\cong \HH_{-n} (\ahom(M,\wti{S}))
\cong\avtate^{n}(M, N)$$
for each $n\geq 1$,
while the first and third isomorphisms are from~\eqref{iso03}
and~\eqref{iso04}.

We claim that the complex $\ahom(\wti{T},\shift^{-1}V)$ is exact.  To see this, 
note that the condition $\opg(\catw)\perp\catv$ implies that
$\aext^{\geq 1}(T_i',V_j)=0$ for all indices $i$ and $j$.
Since $\wti{T}$ is bounded below, a standard argument implies that
$\ahom(\wti{T},V_j)$ is exact for each index $j$, and similarly
it follows that $\ahom(\wti{T},V)$ is exact. We conclude that
$\ahom(\wti{T},\shift^{-1}V)\cong\shift^{-1}\ahom(\wti{T},V)$ is also exact, as claimed.

Now, apply $\ahom(\wti{T},\shift^{-1}(-))$ to the 
degreewise split exact sequence~\eqref{exact25} 
to obtain the next exact sequence
$$0\to\ahom(\wti{T},\shift^{-1}V)\to\ahom(\wti{T},\shift^{-1}\wti{S})\xra{\ahom(\wti{T},\shift^{-1}\eta)}\ahom(\wti{T},Y)\to 0.$$
The exactness of $\ahom(\wti{T},\shift^{-1}V)$ established above
shows that the morphism
$\ahom(\wti{T},\shift^{-1}\eta)$
is a quasiisomorphism.
From \cite[(6.6.b)]{sather:gcac} we know  that the 
first morphism in the following sequence is a quasiisomorphism
$$\ahom(\wti{T},N)\xra[\simeq]{\ahom(\wti{T},\delta)}\ahom(\wti{T},Y)
\xla[\simeq]{\ahom(\wti{T},\shift^{-1}\eta)}\ahom(\wti{T},\shift^{-1}\wti{S}).$$
Combined together, these
yield the first quasiisomorphism in~\eqref{quiso01};  the second one
is dual.
This completes the proof when $n\geq 1$.

For the remainder of the proof, assume that $n<1$ and 
that $\finrescatw=\fincorescatv$. 
Fix a $\catw\opg(\catw)$-hull
\begin{equation} \label{exact211} 
0\to M\to K \to M'\to 0
\end{equation}
that is, an exact sequence in $\cata$ 
with $K\in\finrescatw=\fincorescatv$ and $M'\in\opg(\catw)$;
see Definition~\ref{ab01}. 
We proceed by descending induction on $n$.
The base case $n\geq 1$ has already been established. 
Assuming that the desired isomorphisms hold with index $n+1$,
we have the second isomorphism in the next sequence
$$\waext^n(M,N)\cong\waext^{n+1}(M',N)
\cong\avext^{n+1}(M',N)\cong\avext^n(M,N).$$
The first isomorphism is from Lemma~\ref{dimshft03}\eqref{dimshft03item2},
and the third
isomorphism is from Lemma~\ref{dimshft04}\eqref{dimshft04item1}.
This completes the proof.
\end{proof}

\begin{para} \label{prfC}
\emph{Proof of Theorem~\ref{thmC}.}
We need to check that the categories 
$\catw=\catp_B(R)$ and
$\catv=\cati_{B^{\dagger}}(R)$
satisfy the hypotheses of Theorem~\ref{tatebalance01}.
We have 
$\catp_B(R)\perp\catp_B(R)$ and 
$\cati_{B^{\dagger}}(R)\perp\cati_{B^{\dagger}}(R)$;
see Fact~\ref{disc01}.
The conditions
$\opg(\catp_B(R))\perp\cati_{B^{\dagger}}(R)$
and
$\catp_B(R)\perp\opg(\cati_{B^{\dagger}}(R))$
are from~\cite[(6.16)]{sather:gcac}.
The fact that
$\catp_B(R)$ is closed under kernels of epimorphisms and direct summands, and
that $\cati_{B^{\dagger}}(R)$ is closed under cokernels of monomorphisms and direct summands
is in Fact~\ref{disc01}. 
We have 
$$\ext^{\geq 1}_{\catp_B}(\operatorname{res}\comp{\catp_B(R)},\cati_{B^{\dagger}}(R))
=0 =
\ext_{\cati_{B^{\dagger}}}^{\geq 1}(\catp_B(R),\operatorname{cores}\comp{\cati_{B^{\dagger}}(R)}).
$$
from~\cite[(6.15)]{sather:gcac}.
Finally, when $R$ is noetherian and $C$ is dualizing for $R$, we have
$\operatorname{res}\comp{\catp_B(R)}=
\operatorname{cores}\comp{\cati_{B^{\dagger}}(R)}$
by Lemma~\ref{lemold}.
\qed
\end{para}

\begin{cor} \label{Gbalance02'}
Let $R$ be a commutative ring, and let $M$ and $N$ be $R$-modules
such that $\gpd_R(M)<\infty$ and $\gid_R(N)<\infty$.
For each $n\geq 1$, we have
$$\pmtate^{n}(M, N)\cong\mitate^{n}(M, N).$$
When $R$ is Gorenstein, this isomorphism holds for all $n\in\mathbb Z$.
\end{cor}

\begin{proof}
One readily checks that the categories 
$\catw=\catp(R)$ and
$\catv=\cati(R)$
satisfy the hypotheses of Theorem~\ref{tatebalance01}:
the relative Ext-vanishing follows from the balance
$\pmext\cong\ext\cong\miext$ on $\catm(R)\times\catm(R)$,
and the other hypotheses are standard.
\end{proof}

We conclude with two applications of 
Theorems~\ref{tatewhat01} and~\ref{tatebalance01}.

\begin{thm} \label{tatebalance02}
If $\catw$ and $\catv$ satisfy the hypotheses of
Theorem~\ref{tatebalance01},
then there are containments
$\finrescat{G(W)}\cap\fincorescatv\subseteq \finrescatw$ and
$\fincorescat{G(V)}\cap\finrescatw\subseteq \fincorescatv$.
\end{thm}

\begin{proof}
We verify the first containment; the second one is verified dually.
Fix an object $M\in\finrescat{G(W)}\cap\fincorescatv$.
The object $M$ admits a
$\catw\catx$-hull
$$0\to M\to K\to X\to 0.$$
By assumption, we have $K\in \finrescatw$ and
$X\in\opg(\catw)$.  The condition 
$\catw\perp\opg(\catw)$ from Fact~\ref{sgcsummary}
shows that
$\catw\perp M$, so 
the displayed sequence is $\ahom(\catw,-)$-exact.
Lemma~\ref{dimshft03}\eqref{dimshft03item2} yields the first isomorphism
in the next sequence
$$\watate^0(M,M)\cong\watate^1(X,M)\cong\avtate^1(X,M)=0.$$
The second isomorphism is from Theorem~\ref{tatebalance01},
and the vanishing is from Theorem~\ref{tatewhat02}.
Hence, Theorem~\ref{tatewhat01} implies $\wpd(M)<\infty$,
as desired.
\end{proof}

From this we recover some of the main results of~\cite{sather:mfhdsdm}.

\begin{cor} \label{holm1}
Let $R$ be a commutative ring, and let $C$ be a semidualizing $R$-module.
Let $M$ be an $R$-module.
\begin{enumerate}[\quad\rm(a)]
\item \label{holm1item1}
If $\catgpc\text{-}\pd_R(M)<\infty$ and $\id_R(M)<\infty$, then $\catpc\text{-}\pd_R(M)<\infty$.
\item \label{holm1item2}
If $\gid_R(M)<\infty$ and $\catpc\text{-}\pd_R(M)<\infty$, then $\id_R(M)<\infty$.
\item \label{holm1item3}
If $\catgp\text{-}\pd_R(M)<\infty$ and $\catic\text{-}\id_R(M)<\infty$, then $\pd_R(M)<\infty$.
\item \label{holm1item4}
If $\catgic\text{-}\id_R(M)<\infty$ and $\pd_R(M)<\infty$, then $\catic\text{-}\id_R(M)<\infty$.
\end{enumerate}
\end{cor}

\begin{proof}
We prove part~\eqref{holm1item1}; the other parts are similar or easier.
Assume that $\catgpc\text{-}\pd_R(M)<\infty$ and $\id_R(M)<\infty$.
The finiteness of $\id_R(M)$ implies that $M\in\catbc(R)$, by Fact~\ref{disc01}.
Hence, the condition $\catgpc\text{-}\pd_R(M)<\infty$ 
works with Lemma~\ref{lem0701} to imply that
$\opg(\catpc)\text{-}\pd_R(M)<\infty$. 
Now apply Theorem~\ref{tatebalance02}
with $\catw=\catpc(R)$ and $\catv=\cati(R)$
to conclude that
$\catpc\text{-}\pd_R(M)<\infty$.
\end{proof}

\section*{Acknowledgments}

We are grateful to  R.~Takahashi and D.~White for letting us  include
Lemma~\ref{lemold}.

\providecommand{\bysame}{\leavevmode\hbox to3em{\hrulefill}\thinspace}
\providecommand{\MR}{\relax\ifhmode\unskip\space\fi MR }
\providecommand{\MRhref}[2]{%
  \href{http://www.ams.org/mathscinet-getitem?mr=#1}{#2}
}
\providecommand{\href}[2]{#2}

\end{document}